\documentclass{article}
\usepackage{amsmath,amssymb,amsthm}
 
\usepackage{a4wide}
 
\usepackage{color}
\usepackage{enumerate}
\usepackage{cleveref}
\usepackage{mathrsfs}
\usepackage{url}

\newtheorem{thm}{Theorem}[section]
\newtheorem*{thm*}{Theorem}
\newtheorem*{conj*}{Conjecture}
\newtheorem{cor}[thm]{Corollary}

\newtheorem{lem}[thm]{Lemma}
\newtheorem{prop}[thm]{Proposition}

\theoremstyle{remark}
\newtheorem{rem}[thm]{Remark}

\theoremstyle{definition}

\newtheorem{defn}[thm]{Definition}
\newcounter{claim}[thm]
\newtheorem{problem}[thm]{Problem}

\DeclareMathOperator{\Alt}{Alt}

\DeclareMathOperator{\Sym}{Sym}
\DeclareMathOperator{\Sub}{Sub}
\DeclareMathOperator{\Aut}{Aut}
\DeclareMathOperator{\Out}{Out}

\newcommand{\Inn}{\mathrm{Inn}}

\newcommand{\soc}{\mathrm{soc}}
\newcommand{\AGL}{\mathrm{AGL}}
\newcommand{\GL}{\mathrm{GL}}
\newcommand{\PGL}{\mathrm{PGL}}
\newcommand{\SL}{\mathrm{SL}}
\newcommand{\PSL}{\mathrm{PSL}}
\newcommand{\Sp}{\mathrm{Sp}}

\newcommand{\PSU}{\mathrm{PSU}}

\newcommand{\qsr}{quasi-semiregular}
\newcommand{\qsre}{quasi-semiregular element}
\newcommand{\qsres}{quasi-semiregular elements}

\newcommand{\la}{\langle}

\newcommand{\ra}{\rangle}

\title{Finite permutation groups with quasi-semiregular elements}

\author{Michael Giudici, Luke Morgan and  Cheryl E.~Praeger\\Department of Mathematics and Statistics\\
The University of Western Australia, Perth, Australia
}

\begin{document}

 \maketitle
 \begin{abstract}
 A quasi-semiregular element in a permutation group is an element that has a unique fixed point and acts semiregularly on the remaining points. Such elements were first studied in the context of automorphisms of graphs and occur naturally in many families of permutation groups, such as Frobenius and Zassenhaus groups. They also arise in the context of groups with a strongly $p$-embedded subgroup.
 
 We investigate the question of which finite permutation groups contain quasi-semiregular elements, with particular attention to the primitive permutation groups. We determine  the O'Nan-Scott types of primitive groups that can contain {\qsres} and reduce the question to the affine and almost simple cases. In the almost simple case, we obtain a complete classification when the socle is alternating  or  sporadic.

 \medskip


 \end{abstract}

\section{Introduction}
 
A   permutation group $G$ on a set $\Omega$ is \emph{semiregular} if  the only element of $G$ that fixes a point is the identity. A nonidentity permutation $g\in \Sym(\Omega)$ is   \emph{semiregular} if $\langle g \rangle$ is semiregular. In particular, a semiregular permutation has all of its cycles of the same length. 

\medskip

\noindent {\bf 1.1 Semiregular automorphisms in graph theory.} \quad Semiregular permutations have appeared in many applications of group theory (see  \cite{PabloSurvey} for a good survey). Notably in graph theory, semiregular elements have been well investigated, possibly due to Frucht's method of compact descriptions of graphs \cite{Frucht}.  Maru\v{s}i\v{c} \cite[Problem 2.4]{MarusicPolyCirc} asked if all vertex-transitive graphs admit a semiregular automorphism (see also \cite{JordanPC,Leighton}). This was generalised by Klin \cite{Klin} to what has become known as the Polycirculant Conjecture: `Every non-trivial finite transitive $2$-closed permutation group contains a semiregular element' \cite{7author}. (In particular, the automorphism group of a graph is a $2$-closed permutation group.) Since nonidentity powers of semiregular elements are also semiregular, a group contains a semiregular element if and only if it contains one of prime order. Such an element is therefore a \emph{derangement}, that is, an element without fixed points. A theorem of Jordan from 1872  \cite{Jordan}  shows that every finite transitive permutation group contains a {derangement}.  As such, the study of derangements has always been a central part of permutation group theory. Questions concerning the proportion of derangements in permutation groups and the possible orders of derangements have motivated much work in the field. See for example \cite{Serre}. Further,  connections to graph theory, number theory and algebra have guaranteed the longevity of this topic. Although Jordan's theorem guarantees the existence of a derangement, there is no information about the order of a derangement. A deep theorem of Fein, Kantor and Schacher \cite{FKS}, using the Classification of the Finite Simple Groups, shows that every finite transitive permutation group has a derangement of \emph{prime power order}. However, there do exist transitive groups without derangements of \emph{prime} order. Such groups are termed \emph{elusive}; one example is the Mathieu group $M_{11}$ acting on $12$ points (see \cite[\S2]{7author} for this and other examples). A natural goal is to classify the elusive permutation groups, that is, groups in which every element of prime order has a fixed point. This would also form a plan of attack on the Polycirculant Conjecture -- classify all elusive permutation groups, and then check which are $2$-closed.

\medskip

\noindent {\bf 1.2 Quasi-semiregular automorphisms in graph theory} \quad When studying vertex-transitive graphs, the `local to global' approach has often yielded strong results. Local here refers to a single vertex, and the method uses only knowledge of vertex-stabilisers to derive results on the whole graph or the full automorphism group. A \emph{local} version of semiregularity is the following. A permutation group $G$   is \emph{\qsr} if $G$ has a unique fixed point   and acts semiregularly on the remaining points. A permutation $g\in G$ is \emph{\qsr} if $\langle g \rangle$  is {\qsr}.  If the set acted upon has $n$ points, then as a permutation, a {\qsre} of order $m$ has cycle type $1^1m^{(n-1)/m}$.

Quasi-semiregular elements were first defined by Kutnar, Malni\v{c}, Mart\'{i}nez and Maru\v{s}i\v{c}  \cite{KMMM2013} where they were used to construct a family of graphs that are generalisations of Cayley graphs. Let $\Gamma$ be a graph and let $G\leqslant \Aut(\Gamma)$. We say that $\Gamma$ is an $m$-Cayley graph over $G$ if $G$ acts semiregularly on $V\Gamma$ with $m$-orbits. (So a $1$-Cayley graph is a Cayley graph.)  A graph $\Gamma$ is defined in \cite{KMMM2013} to be a quasi $m$-Cayley graph over a group $G  \leqslant \Aut(\Gamma)$ if $G$ fixes a unique vertex, $\infty$ say, and $G$ acts semiregularly on $V\Gamma\setminus\{\infty\}$ with $m$-orbits. (Note that every nonidentity element of $G$ is  therefore quasi-semiregular.) The authors of \cite{KMMM2013} were interested in finding  families of quasi $m$-Cayley graphs with ``good symmetry properties" and in particular   new strongly regular graphs. The idea has been validated by Williford \cite{Williford} who recently showed that a strongly regular graph with parameters $(65, 32, 15, 16)$ found by Gritsenko \cite{Gritsenko} in 2021 is a quasi $2$-Cayley  graph. Following this discovery, further examples have now been found by Martin and Williford \cite{Martin}.    We also note that under the name `$1$-rotational difference families', {\qsres} have been used to construct designs with various parameters, see \cite[\S16.6]{HCD} and \cite{Burattietal2025}.

 Recall that the Polycirculant Conjecture asks if every vertex-transitive graph has a semiregular automorphism.
The corresponding question `does every vertex-transitive graph admit a {\qsr} automorphism?' is readily seen to have a negative answer: let $\Gamma=C_{2n}$, then each non-trivial element of $\Aut(\Gamma)\cong D_{2n}$  fixes either  $0$ or $2$ points.  Investigations suggest that arc-transitive graphs admitting {\qsr} automorphisms seem to be rather rare. The only  arc-transitive cubic graphs admitting a {\qsr} automorphism are  $K_4$, the Petersen graph and the Coxeter graph   \cite[Theorem 1.1]{FengHujetal}. 
(Note that the order of a {\qsr} automorphism must divide the valency of the graph -- the neighbourhood of a vertex fixed by the automorphism is a union of its cycles.)
Also, if $\Gamma$ is an arc-transitive $4$-valent graph admitting a {\qsr} automorphism, then it is shown in \cite[Theorem 1.1]{FengHujetal} and \cite{YinFeng2021} that $\Gamma$ is a Cayley graph over an abelian group of odd order (and there exist infinitely many examples). In the prime valent case, \cite[Theorem 1.1]{Yinetal2023} shows that a connected arc-transitive graph $\Gamma$ of valency $p\geqslant 5$ admitting a {\qsr} automorphism is either a Cayley graph over a $2$-group, or, a normal cover of a `basic' graph $\overline{\Gamma}$.
For the $5$-valent case, \cite[Theorem 1.2]{Yinetal2023} shows that there are exactly eleven graphs occurring as `basic' graphs. One of the basic graphs is $K_6$, and infinitely many covers are constructed   which are quasi Cayley graphs, \cite[Theorem 1.3]{Yinetal2023}. 
A classification of all vertex-transitive graphs admitting a {\qsr} automorphism is called for in \cite[Problem A]{Yinetal2023}.
Hujdurovi\'{c} \cite[Theorem 1.1--1.3]{Hujdurovic2013} has determined which circulants  admit {\qsr} automorphisms with few orbits; some famous examples are  complete graphs and  Paley graphs. Motivated by the progress on {\qsres} acting on graphs, we aim to make a general description of the kinds of transitive permutation groups that contain {\qsres}.  

\medskip
\noindent {\bf 1.3 Quasi-semiregular elements in permutation groups} \quad 
 It turns out that {\qsres} are present in a rich collection of permutation groups, for example, in Frobenius  and Zassenhaus groups. {  We also find examples  in abstract  group theory. Several families arise from strongly $p$-embedded subgroups and the new concept of subnormalisers yields examples, see Section~\ref{sec:egs-gp theory}.}
 We now turn to general permutation groups. 
Let $G \leqslant \Sym(\Omega)$ be a transitive permutation group and suppose $x\in G$ is a {\qsre} with unique fixed point $\alpha \in \Omega$. 
{  In Lemma~\ref{lem:red to prim} we prove that $x$ induces a {\qsr} permutation on each system of imprimitivity preserved by $G$.}
Hence it makes sense to first consider which \emph{primitive} permutation groups contain {\qsres}. 

The finite primitive permutation groups are classified according to the O'Nan-Scott Theorem. Following the division in \cite{PraegerPrimInc}, there are eight O'Nan-Scott types;  HA (holomorph of an abelian group), AS (almost simple), HS (holomorph of nonabelian simple group), HC (holomorph of a nonabelian composite group), SD (simple diagonal), CD (compound diagonal), PA (product action of AS), TW (twisted wreath). See \cite[Section 3]{PraegerPrimInc} for further description of the types and discussion.
In Section~\ref{sec: reduction} we establish the following result.

\begin{thm}
\label{introthm:prims}
    Suppose that $G$ is a finite primitive permutation group. If $G$ contains a {\qsre}, then $G$ is of type HA, AS, PA, SD or CD.  Furthermore:
    \begin{enumerate}[(a)]
    \item if $G$ is a maximal group of  type SD with $\soc(G)=T^k$, then $G$ contains {\qsres} if and only if $k$ is prime and $|T|$ is coprime to $k$;
    \item if $G$ is a maximal product action  group $H \wr \Sym(\ell)$ (of type PA or CD), then $G$ contains {\qsres} in its action on $\Delta^\ell$ if and  only if $H$ (of type AS or SD) contains {\qsres} in its action on $\Delta$.
    \end{enumerate} 
 \end{thm}

We note that not all SD groups contain {\qsres}, see Remark~\ref{rem:noSDeg}. Theorem~\ref{introthm:prims}(b) shows that  the PA  and CD type primitive groups containing a {\qsre} are   understood   given complete knowledge for primitive groups of type AS and SD respectively (but see Remark~\ref{rem:CDeg}) for CD groups. Further, if an affine group $G\leqslant \AGL(d,q)$ contains an element $g$ that is conjugate to a nonidentity scalar in $\GL(d,q)$, then $g$ is {\qsr}. In addition, there exist affine groups containing {\qsres}  that are not (conjugates of) nonidentity scalars: $\mathrm{ASL}(2,2)\cong \Alt(4)$ is one such example.    In fact,  {\qsres} feature in all $\frac{3}{2}$-transitive affine groups (and in particular in all $2$-transitive affine groups). A nonregular transitive permutation group $G$ on a set $\Omega$ is $\frac{3}{2}$-transitive if  for any $\alpha\in \Omega$, all non-trivial $G_\alpha$-orbits have the same size. 
The theorem below depends on the classification of the finite $\frac{3}{2}$-transitive groups completed in 2019 (see~\cite{LPS2019}).  

\begin{thm}
\label{intro thm:3/2}
Every $\frac{3}{2}$-transitive  affine group contains a {\qsre}.
\end{thm}

On the other hand, there are primitive affine groups with no {\qsres}: for example, the group $3^3 \rtimes \Alt(4)$ acting on $27$ points has no {\qsre}. An important open problem is the following:

\begin{problem}
    Classify the primitive affine groups having no  {\qsres}. 
\end{problem}

We are thus motivated to 
begin a program to classify the primitive almost simple groups $G$ with {\qsres}. In this paper we start  with groups having alternating or sporadic socles.  In forthcoming work, we will consider the groups of Lie type \cite{GMPlieqsres}.  Our first theorem considers the case where the socle of $G$ is an alternating group. (For a group $G$ with subgroup $H$, we denote by $[G:H]$ the set of right cosets of $H$ in $G$ on which $G$ acts transitively by right multiplication.)

\begin{thm}
\label{intro thm:alt or sym}
Let $G$ be a finite almost simple group with $\soc(G) = \Alt(n)$ and let $H$ be a maximal subgroup of $G$ not containing $\Alt(n)$. In the action of $G$ on $[G:H]$, $G$ contains a {\qsre} of prime order $p$ if and only if one of the following holds:
\begin{enumerate}[(1)]
\item $H$ and $p$ are as in one of the lines of Table~\ref{tab:infinite families};
\item $G=\Alt(n)$ and $H$, $p$ and $n$ are as in  one of the lines of Table~\ref{tab:exceps for altn};
\item $n=6$ and, either $p=3$, $|G:H| =10$ and $\Alt(6) \leqslant G \leqslant \Aut(\Alt(6))$; or, $p=5$, $|G:H|=36$ and $G$ contains an outer automorphism of $\Sym(6)$.
\end{enumerate}
\end{thm}

 \begin{table}
     \centering
     \begin{tabular}{c |c | c}
          action of $H$ & structure of $H$  &  conditions  on $p$\\ \hline 
                     $k$-subset stabiliser & $(\Sym(k) \times \Sym(n-k) )\cap G$ &    $ k < p , \quad  p\mid n-k$, and \\ & &    $n\equiv 1 \pmod{4} \text{ if } (p,G)=(2,\Alt(n)) $  \\ 
          partition stabiliser & $(\Sym(p) \wr \Sym(n/p) )\cap G$ &  $p$ odd, \quad  $p \mid n$, \quad $2p\leqslant n\leqslant p^2$  \\ 
                     affine & $ \AGL(1,p) \cap G$ &  $ n=p$, and\\  &&
 $p\neq 7,11,17,23\text{ if }G=\Alt(n)$    \\ 
          projective line &  $\PGL(2,p) \cap G$ &  $n=p+1$, \quad $p\geqslant 5$    
     \end{tabular}
     \caption{Possibilities for $H$  and $p$ occurring in Theorem~\ref{intro thm:alt or sym}(1).}
     \label{tab:infinite families}
 \end{table}

\begin{table}
    \centering
    \begin{tabular}{c|c|c|c|c}
        row & $n$ & $H$ &  $|G:H|$ & $p$\\ \hline
        $1$ & $7$ & $\PSL(3,2)$ & $120$ & $7$ \\
        $2$ &  $8$ & $\AGL(3,2)$ & $120$ & $7$ \\
        $3$ & $9$ & $\mathrm{P}\Gamma\mathrm{L}(2,8)$ & $280$  & $7$ \\
        $4$ &  $11$ & $M_{11}$& $362\,880$ & $11$ \\
        $5$ & $12$ & $ M_{12}$ & $362\,880$ &  $11$ \\
        $6$ & $17$ & $\mathrm{P}\Gamma\mathrm{L}(2,2^4)$ & $10\,897\,286\,400$ & $17$ \\
        $7$ & $23$ & $M_{23}$ & $1\,267\,136\,462\,592\,000$ & $23$  \\
        $8$ & $24$ & $M_{24}$ & $1\,267\,136\,462\,592\,000$ & $23$ \\ 
    \end{tabular}
    \caption{Possibilities for $n$, $H$, $|G:H|$  and $p$ in Theorem~\ref{intro thm:alt or sym}(2).}
    \label{tab:exceps for altn}
\end{table}

In the case that the socle of $G$ is a sporadic simple group, we prove:

\begin{thm}
\label{thmintro: sporadic}
    Let $G$ be an almost simple group with socle a sporadic simple group and let $H$ be a maximal subgroup of $G$. Then $G$ contains a {\qsre} of prime order $p$ in its action on $[G:H]$ if and only if there is a conjugacy class of elements of order $p$ in the row corresponding to $H$ in one of Tables~\ref{tab:M}--\ref{tab:spor3} in Section~\ref{sec:tables}.
\end{thm}

From our classification, we observe that distinct conjugacy classes of {\qsr} subgroups of a  specified prime order are  quite special.

\begin{cor}
\label{cor: more than one class of qsrs}
    Suppose that $ G$ is a primitive group of degree $m$ with socle an alternating or sporadic group. If $G$ has at least two distinct conjugacy classes of {\qsr} subgroups of prime order $p$, then $G$, $p$ and $m$ are as in one of the lines of Table~\ref{tab: >1 qsr classes}.
\end{cor}

    \begin{table}
        \centering
        \begin{tabular}{c|c|c|c}
            $G$ & $p$ & $m$ & Notes\\ \hline
             $\Alt(n) \leqslant G \leqslant \Sym(n)$ & $p$ & ${n!}/({(p!)^{n/p}(n/p)!})$ & $p\mid n$ and $n < p^2$ \\
             $HS$ & $5$ & $176$ & Two representations 
             \\
             $HS.2$ & $5$ & $22\, 176$ \\\
             $McL \leqslant  G \leqslant McL.2$ & $5$ & $299 \,376$ \\
             $J_4$ & $11$ & $271\, 649\, 045\, 348\, 352$ \\
        \end{tabular}
        \caption{Primitive groups with alternating or sporadic socles and at least two conjugacy classes of {\qsr} subgroups of order $p$.}
        \label{tab: >1 qsr classes}
    \end{table}

We would like to know whether this behaviour is indeed rare, and we pose:

\begin{problem}
    Classify the transitive permutation groups with at least two conjugacy classes of {\qsr} subgroups of   the same prime order.
\end{problem}

\medskip

\noindent {\bf Acknowledgements:} \quad We thank Gabriel Verret for useful preliminary discussions about the existence of \qsres.  We also thank Chris Parker for a helpful discussion on  strongly  $p$-embedded subgroups, {   and Gunter Malle for helpful discussions about subnormalisers.}  This research was supported by the Australian Research Council Discovery Project grant DP230101268.

\section{Preliminaries}

This section contains several group theoretic observations which are useful when studying groups containing {\qsres}.   

{  
\begin{lem}\label{lem:red to prim}
Let $G$ be a finite transitive permutation group on a set $\Omega$ and suppose that $x\in G$ is {\qsr}. If $\Sigma$ is a system of imprimitivity preserved by $G$, then the permutation $x^{\Sigma}$ induced on $\Sigma$ by $x$ is {\qsr}.
\end{lem}
\begin{proof}
     Let $\alpha$ be the unique point of $\Omega$ fixed by $x$. Then $x$ fixes the part $B\in \Sigma$ with $\alpha \in B$. Since $\alpha$ is the only point fixed by $x$, and $x$ acts semiregularly on the remaining points, $B$ is a union of orbits of $\langle x\rangle$. If $x$ has order $m$, then  $|B|\equiv 1 \pmod{m}$. If $B'\neq B$ is another part of $\Sigma$, then $|B|=|B'|$ so $B'$ cannot be a union of cycles of $x$, and thus $x$ does not fix~$B'$.   Moreover the cycle containing $B'$ of $x^\Sigma$, the permutation induced on $\Sigma$ by $x$,  has length $m$ since otherwise $(|B'|,m)>1$.  It follows that $x^\Sigma$ is a {\qsre} of $G^{\Sigma}$, the permutation group induced by $G$ on $\Sigma$. 
\end{proof}

\begin{rem}\label{rem: subs vs systs}
    Each system of imprimitivity $\Sigma$ preserved by $G$ corresponds to an overgroup of $G_\alpha$, namely the group $K=G_B$ where $B\in \Sigma$ is the block containing $\alpha$. Hence Lemma~\ref{lem:red to prim} can be interpreted as follows in terms of subgroups and coset actions: if $x\in G_\alpha$ and $x$ is {\qsr} on $\Omega=[G:G_\alpha]$, then for any subgroup $G_\alpha \leqslant K \leqslant G$, we have that $x$ is {\qsr} on $[G:K]$. 
\end{rem}
}

\begin{lem}\label{lem: qs and sylow} 
Suppose that $G$ is a finite transitive permutation group and a stabiliser $H=G_\alpha$ contains a {\qsr} {  $p$-element, for some prime $p$}. Then the following all hold: 
\begin{enumerate}[(1)]
\item $H$ contains a  Sylow $p$-subgroup of $G$;
\item  if $S$  is a Sylow $p$-subgroup of $G$ contained in $H$, then $N_G(S) \leqslant N_G(H)$;
\item if $N_G(H) =H$, then $N_G(S) = N_H(S)$.
\end{enumerate}
\end{lem}
\begin{proof}
Let $n=|G:H|$. If $x\in H$ is a {\qsr} {  $p$-element}, then $x$ fixes a unique point and has orbits of a fixed length {  $p^i$, with $i\geqslant 1$, on the remaining points. Hence $n-1 = kp^i$ for some integer $k$. In particular, $n$} is coprime to $p$ and so $H$ contains a Sylow $p$-subgroup $S$ of $G$. Without loss of generality, we may assume that $x \in S$. In particular, for any $g\in N_G(S)$ we have $x \in S  = S^g \leqslant H^g$. Since $x$ is {\qsr}, $H=H^g$ and so $g\in N_G(H)$.
\end{proof}

\begin{lem}
\label{lem: x is qsres}
    Let $G$ be a finite group acting transitively
on a set $\Omega$ and let $\alpha\in \Omega $. For $g\in G$, let $\pi(g)$ denote the number of  points in $\Omega$ that are fixed by $g$. Then 
    $$\pi(g) = \frac{|G:G_\alpha |\cdot|g^G \cap G_\alpha|}{|g^G|}$$
    and if $g^G \cap G_\alpha$ is the disjoint union   $g_1^{G_\alpha} \dot\cup \cdots \dot\cup g_n^{G_\alpha}$, then
    \[\pi(g) = \sum_{i=1}^n |C_G(g) : C_{G_\alpha}(g_i)|= \sum_{i=1}^n |C_G(g_i ) : C_{G_\alpha}(g_i)|. \]
\end{lem}
\begin{proof}
We include a proof for completeness.

Let $\chi$ be the permutation character of the $G$-action on the set of  cosets of $G_\alpha$. Then $\chi$ is  induced from the trivial character $\varphi$ of $G_\alpha$. Furthermore, for an element $g\in G$, $\chi(g)$ is the trace of the matrix of $g$, which is the number of fixed points of $g$. From the formula for induced characters \cite[Definition (5.1)]{Isaacs}, we have
\[
\chi(g) =\frac{1}{|G_\alpha|} \sum_{t\in G, tgt^{-1} \in G_\alpha} \varphi(tgt^{-1})
\]
Letting $t_1,\ldots,t_n$ be a transversal for $C_G(g)$ in $G$ we have $g^G = \{g^{t_1},\ldots,g^{t_n}\}$, and we may order the $t_i$ such that $g^{t_i} \in G_\alpha$ for $i=1,\ldots, k$, where $k=|g^G\cap G_\alpha|$. Then 
\[
\chi(g) =\frac{1}{|G_\alpha|} |C_G(g)| (\varphi(g^{t_1})+\cdots+\varphi(g^{t_k})) = \frac{|G : G_\alpha| \cdot |g^G \cap G_\alpha|}{|g^G|} 
\]
where we use the fact that $\varphi(g^{t_i})=1$ for $g^{t_i} \in G_\alpha$ since $\varphi$ is the trivial character of $G_\alpha$.
\end{proof}

\begin{cor}
\label{cor: qsr iff cent condition}
    Let $G$ be a  finite group acting transitively on a set $\Omega$, and let $g\in G_\alpha$ {  have prime order}, where $\alpha \in \Omega$. Then $g$ is {\qsr} if and only if $g^G \cap G_\alpha = g^{G_\alpha}$ and $C_G(g) = C_{G_\alpha}(g)$.
\end{cor}

{  Note that the primality condition is necessary. For example, $g=(1,2,3,4)(5,6)\in \Sym(7)$ satisfies the conditions on $g$ in Corollary~\ref{cor: qsr iff cent condition} (with $\Omega$ a set of $7$ points), but $g$ is not {\qsr}.}

The following theorem is due to W.~A.~Manning \cite[Theorem XIV]{Manning} (see also \cite[Lemma 2.1]{CherylOnManning} and  \cite[Theorem]{AlperinOnManning}). Our version uses modern terminology,   see \cite[Lemma 2.23]{PraegerSchneider}  for the statement below and a detailed proof. 

\begin{thm}
    Let $G$ be a finite transitive permutation group on $\Omega$, let $H=G_\alpha$ and let $K\leqslant H$. If the set of $G$-conjugates of $K$ which are contained in $H$ form $t$-conjugacy classes $\mathcal C_1$, \ldots, $\mathcal C_t$ of $H$, then $K$ fixes $\sum_{i=1}^t\frac{|N_G(K_i)|}{|N_H(K_i)|}$ points of $\Omega$, where $K_i \in \mathcal C_i$ for $1\leqslant i \leqslant t$. In particular, if $t=1$, then $K$ fixes $|N_G(K):N_H(K)|$ points of $\Omega$.
\end{thm}

{   This result yields a  criterion for \qsres\ of prime order, similar to that in Corollary~\ref{cor: qsr iff cent condition}, in terms of the subgroups they generate.}

\begin{cor}
\label{cor: normalisers equals implies qsr}
  Suppose $G$ is a finite transitive group with point-stabiliser $H$. Let $x\in H$ have prime order and let $K=\la x \ra$. 
Then $x$ is a {\qsre} if and only if 
\begin{enumerate}
    \item[(1)] $N_G(K) = N_H(K)$, and 
    \item[(2)] $K^G \cap H = K^H$. 
    \end{enumerate}
\end{cor}

\section{Examples from group theory}
\label{sec:egs-gp theory}

As mentioned above, {\qsres} turn up in several established families of permutation groups, and have connections with other notions in group theory. 

\subsection{Strongly $p$-embedded subgroups}

A subgroup $H$ of a finite group $G$ is \emph{strongly $p$-embedded} if $p \mid |H|$ and $p\nmid|H\cap H^g|$ for any $g\notin H$. 
Let $G$ act by right multiplication on the set $[G:H]$ of right cosets of $H$ and let $x\in H$ have order $p$. Then $x \notin H\cap H^g$ for any $g\notin H$, and so $x$ is {\qsr} in the action of $G$ on $[G:H]$. In particular, every element of $H$ of order $p$ is {\qsr}. For this reason our attention is drawn to finite groups $G$ with strongly $p$-embedded subgroups. If $G$  has a  cyclic Sylow $p$-subgroup $S$, then $N_G(S_0)$ is strongly $p$-embedded in $G$, where $S_0$ is the unique subgroup of $S$ (and therefore of $N_G(S_0)$) of order $p$. As stated in \cite[pg.~798]{ParkerStroth}, ``there is thus no prospect of listing all such groups'', yet much can be said when Sylow $p$-subgroups are non-cyclic. 
The following theorem is an amalgamation of \cite[Theorem 7.6.1]{GLS3} and \cite[Theorem 7.6.2]{GLS3} and relies on the Classification of the Finite Simple Groups. See also \cite[Proposition 2.5]{ParkerStroth}. For a finite group $G$, $F^*(G)$ denotes the \emph{generalised Fitting subgroup} (the subgroup generated by the Fitting subgroup and all quasisimple subnormal subgroups).

\begin{thm}
\label{thm: st p embedded}
    Suppose that $p$ is a prime, $G$ is a  group such that $K:=F^*(G)$ is simple. Let $P$ be a Sylow $p$-subgroup of $G$. Suppose that $P$ is not cyclic and that $H$   is  a strongly $p$-embedded subgroup of  $G$ with $P\leqslant H$.
    Then $K$ contains a strongly $p$-embedded subgroup $H_0$, $H_0$ is a maximal subgroup of $K$ and $p$, $K$ and $H_0$ are as in Table~\ref{tab:st p embedded}.
\end{thm}

\begin{cor}
    Let $p$,   $K$ and $H_0$ be as in  Table~\ref{tab:st p embedded}. Then  $K$ acts primitively on the set of cosets of $H_0$ and every element of order $p$ in $H_0$ is {\qsr}.
\end{cor}

We remark that the only  examples in Table~\ref{tab:st p embedded} with socle an alternating group appear in case (b), and these examples are identified in  Theorem~\ref{intro thm:alt or sym}(1), occurring  in row 2 of Table~\ref{tab:infinite families}. Comparing Tables~\ref{tab:infinite families} and~\ref{tab:st p embedded}, we observe that examples arising  from strongly $p$-embedded subgroups appear to constitute a small minority of the primitive groups containing {\qsres}. In all other cases appearing in Table~\ref{tab:st p embedded}  the socle is a group of Lie type or a sporadic simple group.  The examples of  sporadic groups are visible in Tables~\ref{tab:spor1}--\ref{tab:spor3} in Section~\ref{sec:tables}.  For example, considering $G=J_4$, we see that $G$ has a strongly $11$-embedded subgroup, leading to {\qsres} in its action on the set of cosets of $H=11^{1+2}_+:(5 \times 2S_4)$. On the other hand, Table~\ref{tab:spor1} shows that $G$ also has {\qsres} of order $5$ and $7$ when $H=2^{3+12}.(S_5 \times L_3(2))$ and {\qsres} of orders $31$ and $23$ when $H= L_2(32).5$ and $H=L_2(23).2$, respectively. Thus our results give examples where point-stabilisers are neither normalisers of cyclic Sylow subgroups nor  strongly $p$-embedded subgroups.

\begin{table}
    \centering
    \begin{tabular}{c |c|c|c | c}
 Case       & $p$ & $K$                 & $H_0$  & Notes\\ \hline
      (a)   & $p$ & $\PSL(2,p^{a+1})$, $\PSU(3,p^a)$   & $N_K(Q)$ & $a\geqslant 1$ \\
      & $2$ &   ${}^2B_2(p^{2a+1})$ & $N_K(Q)$ & $a\geqslant 1$ \\
      & $3$ & $\PSL(2,8)$  & $N_K(Q)$ & $G= \PSL(2,8).3$ \\ 
      & $3$ & ${}^2G_2(p^{2a+1})$  & $N_K(Q)$ &$a\geqslant 1$ \\ 
      & $5$ & ${}^2B_2(32)$ & $N_K(Q)$ & $G={}^2B_2(32).5$ \\  
      (b)   & $>3$ & $\Alt(2p)$ & $(\Sym(p)\wr C_2)\cap K$  & $|G/K|\leqslant 2$ \\
      (c)   & $3$ &  $\PSL(3,4)$ & $N_K(Q) \cong \PSU(3,2)$ & $G/K$ is a $2$-group\\
       (d)  &  $3$ & $M_{11}$  & $N_K(Q) = QS$,\quad $|S|=16$ & $G=K$ \\
         (e)& $5$ & ${}^2F_4(2)'$  & $N_K(Q) = QS$,\quad  $S= \mathbb Z_4\circ \SL(2,3)$ & $|G/K| \leqslant 2$\\
       (f)  & $5$ & $McL$ & $N_K(Q)=QS$,\quad $S=\mathbb Z_3 .\mathbb Z_8$ & $|G/K| \leqslant 2$\\
       (g)  & $5$ &  $Fi_{22}$ & $\Aut(D_4(2))$ & $|G/K| \leqslant 2$\\
      (h)   & $11$ & $J_4$ & $N_K(Q) = QS$,\quad $S = \mathbb Z_5 \times \GL(2,3)$ & $G = K$ \\
    \end{tabular}
    \caption{Examples from strongly $p$-embedded subgroups, where $Q$ is a Sylow $p$-subgroup of $K$.}
    \label{tab:st p embedded}
\end{table}

{  
\subsection{Subnormalisers}

Examples of {\qsres} also arise from finite groups for which the subnormaliser of an element of prime order is a proper subgroup: the \emph{subnormaliser} of an element $x\in G$ is defined to be 
$$
\mathrm{Sub}_G(x):=\langle g\in G \mid \langle x \rangle\text{ is subnormal in }\langle x,g\rangle \rangle.
$$
The following connection between {\qsres} and subnormalisers was suggested to us by Gunter Malle.

\begin{thm}\label{l:subn}
Let $G$ be a transitive permutation group on a set $\Omega$ and let $x\in G_\alpha$ have prime order, where $\alpha \in \Omega$. Then $x$ is {\qsr} on $\Omega$ if and only if $\mathrm{Sub}_G(x) \leqslant G_\alpha$.
\end{thm}
\begin{proof}
    Suppose that $x\in G $ is of prime order and is {\qsr} with unique fixed point $\alpha$. Then by Lemma~\ref{lem: qs and sylow}, $G_\alpha$ contains a Sylow $p$-subgroup of $G$, and by Corollary~\ref{cor: qsr iff cent condition}, $G_\alpha$ contains $C_G(x)$ and $x^G \cap G_\alpha = x^{G_\alpha}$. By \cite[Corollary 2.10]{Malle}, $\Sub_G(x)$ is the smallest subgroup of $G$ with these properties (by inclusion), and hence $\Sub_G(x) \leqslant G_\alpha$. 

Now assume that $\Sub_G(x) \leqslant G_\alpha$. Since $G_\alpha$ is a core-free subgroup of $G$, so is $\Sub_G(x)$. In particular, $G$ acts faithfully on $[G:\Sub_G(x)]$. By \cite[Corollary 2.10]{Malle} we have $C_G(x) \leqslant \Sub_G(x)$ and  $x^G \cap \Sub_G(x) = x^{\Sub_G(x)}$. Hence Corollary~\ref{cor: qsr iff cent condition} shows that $x$ is {\qsr} in the action of $G$ on $[G:\Sub_G(x)]$. Applying Lemma~\ref{lem:red to prim} with $\Sigma$ corresponding to the inclusion $\Sub_G(x) \leqslant G_\alpha$ (see Remark~\ref{rem: subs vs systs}), we see that $x$ is {\qsr} in the action of $G$ on $[G:G_\alpha]=\Omega$.
\end{proof}

We observe that $x \in G$ may be {\qsr} in several actions of $G$. For example, let $G=M_{12}$ and let $H=M_{11}$ and $K=\PSL(2,11)$. Let $x\in G$ have order $11$, so $\Sub_G(x)\cong 11:5$ is contained in (conjugates of) $H$ and $K$. Thus $x$ is {\qsr} on both $[G:H]$ and $[G:K]$ (see also Table~\ref{tab:spor2}). This situation is summarised in the following.

\begin{cor}\label{c:subn}
    Suppose that $G$ is a finite group, $x\in G$ has prime order and $\Sub_G(x)$ is a proper core-free subgroup of $G$. Then $x$ is {\qsr} in the action of $G$ on $[G:H]$ for any proper subgroup $H$ of $G$ with $\Sub_G(x) \leqslant H $.
\end{cor}

\begin{proof}
    Let $G$, $x$ and $\Sub_G(x)$ be as in the statement and set $\Omega=[G:\Sub_G(x)]$. Then $G$ acts faithfully on $\Omega$ and by Theorem~\ref{l:subn}, $x$ is {\qsr}. If $H$ is a proper subgroup of $G$ with $\Sub_G(x) \leqslant H$, then with $\Sigma=[G:H]$, Lemma~\ref{lem:red to prim} shows that $x^{\Sigma}$ is {\qsr}.
\end{proof}

Malle's motivation for studying the subnormaliser relates to  a new local-global conjecture of Moret\'{o} and Rizo on values of  irreducible characters of finite groups \cite{MoretoRizo} (see also \cite{MoretoNavarroRizo}). Investigating the validity of this conjecture, Malle showed that the subnormaliser is proper for a particular kind of $p$-element:  a $p$-element of a finite group $G$ is \emph{picky} if it lies in a unique Sylow $p$-subgroup of $G$. For example, if $G$ has cyclic Sylow $p$-subgroups, then any generator of a Sylow $p$-subgroup is picky. Malle  \cite[Corollary 2.7]{Malle} proved that $x$ is picky if and only if $\mathrm{Sub}_G(x)=N_G(P)$ where $P$ is a Sylow $p$-subgroup of $G$ containing $x$. Thus if $x$ of order $p$ is picky and $N_G(P)$ is core-free in $G$ where $P$ is the unique Sylow $p$-subgroup of $G$ containing $x$, 
then Corollary~\ref{c:subn} shows that $x$ is {\qsr} in the action of $G$ on $[G:H]$ for any overgroup $H <G$ of $N_G(P)$.
}

These observations in the case of picky elements yield several explicit families of groups with \qsres\ in \cite[Section 3]{Malle} and \cite[Section 5]{Malle}. One large class of examples are  `regular unipotent' elements (of prime order) in finite groups of Lie type which are {\qsr} when the stabiliser $H$ is a parabolic subgroup. We note that a group contains a picky $p$-element if and only if it does not contain a redundant Sylow $p$-subgroup, as in \cite{MarotiMartinezMoreto}. From our investigations we see that the picky elements {  of prime order} form a proper subclass of {\qsres} of prime order. Indeed, an element such as $(1,2)(3,4)(5,6)\in\Sym(7)$ is {\qsr} in the usual action on $7$ points, but lies in at least three Sylow $2$-subgroups, and hence is not picky. {  Considering arbitrary $p$-elements, the situation is not so clean: an element such as $x:=(1,2,3,4)(5,6)$ is picky in $G:=\Alt(7)$, but $x$ is not {\qsr} in the action on $[G:\mathrm{Sub}_G(x)]$, where $ \mathrm{Sub}_G(x)\cong D_8$. This example also shows that the condition that $|x|$ is prime cannot be dropped from the assumptions of Theorem~\ref{l:subn}.}

\section{Affine groups}
\label{sec: affine groups}

\subsection{$2$-transitive affine groups}
The $2$-transitive affine groups were classified by Huppert \cite{Huppert} in the soluble case, and by Hering \cite{Hering} in the insoluble case.  There are four infinite classes, and several families of exceptions that occur for certain degrees less than $4096$.  For a complete list, we refer to \cite[Appendix 1]{Liebeck1984} and \cite[\S7.3]{Cameron}.

\begin{thm}
\label{thm: 2-trans affine thm}
Suppose that $G$ is a $2$-transitive  affine group of degree at least $3$. Then $G$ has a {\qsre}.
\end{thm}
\begin{proof}
First suppose that $G$ is soluble. Apart from exceptional examples occurring for degrees $3^2$, $5^2$, $7^2$, $11^2$, $23^2$, $3^4$, we may assume that $G \leqslant \mathrm{A}\Gamma\mathrm{L}(1,q)$    \cite[XII, 7.3]{HuppertBlackburn}. For the exceptional examples, we confirm the validity of the  theorem with {\sc Magma} \cite{Magma} using the database of primitive permutation groups \cite{Colva}. In the general case,  write $G = V G_0$, where $V=\mathbb F_q$ and $G_0 \leqslant \Gamma\mathrm{L}(1,q)$. Let $p$ and $f$ be such that $q=p^f\geqslant3$. Then $\Gamma\mathrm{L}(1,q) \cong C_{p^f-1}\rtimes C_f$. Now $|G_0 : G_0\cap \GL(1,q)| \leqslant f$, and by $2$-transitivity, $|G_0|$ is divisible by $p^f-1$. It follows that $|G_0 \cap \GL(1,q)| > 1$.  Since all nonidentity elements of $\GL(1,q)$ are {\qsr}, it follows that $G$ contains a {\qsre}. 

Now suppose that $G=VG_0$ is insoluble (and hence $G_0$ is insoluble). For the groups $G$ that do not belong to an infinite family, we use {\sc Magma} (as described above) to confirm the truth of the result. Thus we may assume  that $G_0$ contains an insoluble normal subgroup $H$ isomorphic to $\SL(d,q)$, $\Sp(2m,q)$ or $\mathrm G_2(q)$. In the first case, an extension field subgroup  $\GL(1,q^d) \cap H\neq 1$ contains {\qsres} of prime order. In the second case, there is an extension field subgroup $X \cong \GL(1,q^{2m})$ embedded in $\GL(2,q^m)$.  Since $\SL(2,q^m)=\Sp(2,q^m)$, the subgroup  $X \cap \Sp(2,q^m)\neq 1$ embeds into $\Sp(2m,q)$. Since elements of $X \cap \Sp(2,q^m) \leqslant \Sp(2m,q)$ act semiregularly on $V$, there exist {\qsres} of prime order in this case also.

Finally, assume that $H=\mathrm G_2(q)$ with $q$  a power of $2$ and $|V|=q^6$. We use Wilson's notation \cite[4.4.3]{WilsonBook}  for $\mathrm G_2(q)$. Here $H$ acts on the $8$-dimensional Octonion algebra over $\mathbb F_q$ spanned by $x_1,\ldots,x_8$, and the space $V = \langle x_4 + x_5 \rangle^\perp / \langle x_4 + x_5 \rangle$. 
Let $V_1=\langle x_4+x_5\rangle$,  $U=\langle V_1, x_1,x_6,x_7\rangle$ and $W=\langle V_1, x_2, x_3, x_8 \rangle$. Then $V = U/V_1 \oplus W/V_1$ and the subgroup $X$ of $H\cong \mathrm G_2(q)$ preserving this decomposition satisfies $X \cong \SL(3,q) :2$. Elements of the derived subgroup $X'$ act naturally on $U$ and dually on~$W$. In particular, there are elements of order dividing $\frac{q^3-1}{q-1}$ that act semiregularly on the non-zero vectors of $U$, and therefore they are also semiregular on the non-zero vectors of $W$; hence these elements are also semiregular on the non-zero vectors of $V$. Thus $X$, and hence $G$, contains {\qsres}.
\end{proof}

\subsection{$\frac{3}{2}$-transitive affine groups}
\label{sec: 3/2-trans}

Recall from the introduction that $G\leqslant \Sym(\Omega)$ is $\frac{3}{2}$-transitive if $G$ is a nonregular transitive group such that for any $\alpha \in \Omega$, all   non-trivial $G_\alpha$ orbits  have the same size. This class of permutation groups was first defined by Wielandt \cite[\S10]{Wielandt}, who proved that a $\frac{3}{2}$-transitive group is either primitive or Frobenius. Following work of Passman \cite{Passman1,Passman2} (who classified the soluble $\frac{3}{2}$-transitive groups), Bamberg, Giudici, Liebeck, Praeger, Saxl \cite{Bambergetal} and Giudici, Liebeck, Praeger, Saxl, Tiep \cite{Giudicietal}, the full classification was finally completed in 2019 by Liebeck, Praeger and Saxl \cite{LPS2019}.

Here, we focus on the affine case. The following theorem (which includes the aforementioned works) is quoted from \cite{LPS2019}. In the statement, $S_0(q)$ is the subgroup of $\GL(2,q)$ consisting of the $4(q-1)$ monomial matrices of determinant $\pm1$. 

\begin{thm}[{\cite[Corollary 2]{LPS2019}}]
\label{thm: 3on2 thm}
Let $G$ be a $\frac{3}{2}$-transitive group of  affine type and  degree $p^d$, where $p$ is a prime.  Suppose that $G=VG_0$ with $V$ a  regular normal subgroup of $G$. Then at least one of the following holds:
\begin{enumerate}[(1)]
    \item $G$ is $2$-transitive;
    \item $G$ is a Frobenius group;
    \item $G_0\leqslant \Gamma\mathrm{L}(1,p^d)$;
    \item $G_0=S_0(p^{d/2})$ with $p$ odd;
    \item $G_0$ is soluble and $p^d=3^2,5^2,7^2,11^2,17^2$ or $3^4$;
    \item $\SL(2,5) \unlhd G_0\leqslant \Gamma\mathrm{L}(2,p^{d/2})$ where $p^{d/2}=9,11,19,29$ or $169$.
\end{enumerate}
\end{thm}

By considering the individual groups in the theorem above, we find that all contain {\qsres}, proving the theorem mentioned in the introduction, which we restate for convenience.

\smallskip

\noindent {\bf Theorem}~\ref{intro thm:3/2}.     Let $G$ be a $\frac{3}{2}$-transitive group of  affine type and  degree $p^d$. Then $G$ contains a {\qsre}.
\smallskip
\begin{proof}
If $G$ is $2$-transitive the result follows from Theorem~\ref{thm: 2-trans affine thm}. If $G$ is a Frobenius group, then as observed in the introduction, $G$ contains a {\qsre}.

Suppose that case 3~holds so $G_0\leqslant \Gamma\mathrm{L}(1,p^d)=\GL(1,p^d)\rtimes C_d$. If $G_0 \cap \GL(1,p^d) \neq 1$ then $G_0$ contains a nonidentity scalar, which is {\qsr}. Otherwise, $G_0\cap \GL(1,p^d)=1$, so $G_0$ is cyclic of order dividing $d$. Also $G_0$ leaves invariant no proper nontrivial subspace of $\mathbb{F}_p^d$,  since $G$ is  primitive. In particular, $d$ is not a power of $p$, and hence there is a prime $r\mid |G_0|$, with $r\ne p$. Thus the Sylow $r$-subgroups of $G_0$ are Sylow $r$-subgroups of $G$. Since $G$ is primitive, $G_0=N_G(\langle x \rangle)$, where $x$ is an element of order $r$ in $G_0$, and hence $x$ is {\qsr} by Corollary~\ref{cor: normalisers equals implies qsr}.

In cases (4) and (6), $-I \in G_0$  and $-I$ is {\qsr}; while if case (5)~holds, a calculation using {\sc Magma} and the database of primitive groups \cite{Colva} confirms the result.
\end{proof}

\section{A reduction theorem}
\label{sec: reduction}

In this section we first consider finite permutation groups  that leave invariant a product structure on the underlying set. This allows us to settle the question of existence of {\qsres} for many types of primitive groups, and essentially reduces the question to almost simple groups and affine groups.

\begin{thm}\label{t:pa}
Let $G= \Sym(k)\wr\Sym(\ell)$ in its natural action on $[k^\ell]$, where $k\geqslant3$ and $\ell\geqslant2$. Then $G$ contains {\qsres}. Moreover, if $g\in G$ is a {\qsre} of prime order, then  $g=(h_1,\ldots,h_\ell)\in \Sym(k)^\ell$ and   $h_i\in \Sym(k)$ is  {\qsr} for all $i$ such that $1\leqslant i\leqslant \ell$.
\end{thm}
\begin{proof}
    Set $H=\Sym(k)$. We view elements of the set   $[k^\ell]$   as tuples $(\alpha_1,\ldots,\alpha_\ell)$ with $\alpha_i\in [k]$ for each $i$. We write elements of $G$ as products $h\sigma $ where $h\in H^\ell$ and $\sigma \in \Sym(\ell)$.  The elements of $G$  act on $[k^\ell]$ via the following rule. For $x=(\alpha_1,\ldots,\alpha_\ell)\in [k^\ell]$ and $h=(h_1,\ldots,h_\ell)\in H^\ell$ we have $x^h :=  (\alpha_1^{h_1},\ldots,\alpha_\ell^{h_\ell})$ and $x^\sigma  = (\alpha_{1\sigma^{-1}},\ldots,\alpha_{\ell\sigma^{-1}})$. Furthermore, $\sigma \in \Sym(\ell)$ acts by conjugation on $ H^\ell$ via $h^\sigma  := (h_{1\sigma^{-1}},\ldots,h_{\ell\sigma^{-1}})$.

    For each $i=1,\dots, \ell$, let  $h_i \in H$ be quasi-semiregular with unique fixed point $\alpha_i \in [k]$ (note that such an $h_i$ exists since $H=\Sym(k)$ and $k\geqslant3$). Define $h=(h_1,\ldots,h_\ell)$. Then $h$ fixes $x:=(\alpha_1,\ldots,\alpha_\ell)\in [k^\ell]$.  Let $y = (\beta_1,\ldots,\beta_\ell) \in  [k^\ell]$  such that $y \neq x$. Then there exists $i$ such that  $\beta_i\neq \alpha_i $, and hence $\beta_i^{h_i} \neq \beta_i$ since $\alpha_i$ is the unique point of $[k]$ fixed by $h_i$. Thus $y^h \neq y$ and so  $h$ is a {{\qsre}} in $G$.

    Conversely suppose that $g\in G$ is a {\qsre} of prime order $p$. After conjugating by some element of $G$, we may assume that     $g$ fixes the point $x=(\alpha,\ldots,\alpha)$ for some $\alpha \in [k]$. Write $g=h\sigma$ with $h=(h_1,\ldots,h_\ell) $ and note that $\alpha^{h_i} = \alpha$ for $1\leqslant i \leqslant\ell$. Since $g^p=1$ we have $\sigma^p=1$. In the action of $\sigma$ on $\ell$ points, suppose that $\langle \sigma\rangle$ has $a$ orbits of size $p$ and $b$ orbits of size $1$ (so $ap+b=\ell$). 
    Suppose that $a\geqslant 1$. After relabelling, we may assume  that $\sigma$ contains the $p$-cycle $(1,\ldots,p)$ (on $\ell$ points). 
Now we have:  
\[
1=g^p=h\sigma h \sigma \ldots h\sigma =h h^{\sigma^{-1}} \ldots h^{\sigma^{-(p-1)}}   \sigma^p = h h^{\sigma^{-1}} \ldots h^{\sigma^{-(p-1)}} 
\]
and we note that, for all $i$ such that $1\leqslant i\leqslant p-1$
$$
h^{\sigma^{-i}}=(h_{1 \sigma^i},\ldots,h_{\ell\sigma^{i}})=(h_{1+i},\ldots,h_p,h_1,\ldots,h_i,h_{(p+1)\sigma^i},\ldots,h_{\ell \sigma^i}).
$$
Thus the first entry of the $\ell$-tuple $h h^{\sigma^{-1}} \cdots h^{\sigma^{-(p-1)}}$ is $h_1h_2\cdots h_p$. Since $h h^{\sigma^{-1}} \cdots h^{\sigma^{-(p-1)}} = g^p = 1$, it follows that $h_1h_2\cdots h_p=1$.
Pick  $\beta \neq \alpha$,  set 
    \[y= (\beta,\beta^{h_1},\beta^{h_2},\ldots,\beta^{h_1h_2\cdots h_{p-1}},\alpha,\ldots,\alpha)\]
    and note that  
\begin{align*}
       y^{h\sigma} &= (\beta^{h_1},\beta^{h_1h_2},\ldots,\beta^{h_1 h_2 \cdots h_p},\alpha^{h_{p+1}},\ldots,\alpha^{h_\ell})^\sigma \\
       & = (\beta^{h_1\cdots h_p},\beta^{h_1},\ldots,\beta^{h_1 \cdots h_{p-1}},\alpha,\ldots,\alpha).
\end{align*}
Since $\beta^{h_1\cdots h_p}=\beta$, we have $y^g = y$. Now $\beta \neq \alpha$ so $y\neq x$, and this is a contradiction to $g $ being {\qsr}. It follows that $a=0$ so that $\sigma =1$ and $h=(h_1,\ldots,h_\ell) \in \Sym(k)^\ell$. 
Finally, since $h$ is {\qsr}, for any $\beta \neq \alpha$ we have $(\beta,\alpha,\ldots,\alpha)\neq (\beta,\alpha,\ldots,\alpha)^h = (\beta^{h_1},\alpha,\ldots,\alpha) $ so that $\beta^{h_1}\neq  \beta$. Thus, $h_1$ is {\qsr}. Similarly $h_i$ is {\qsr} for each $i$.
\end{proof}

\begin{rem}
    If $g\in \Sym(k)\wr\Sym(\ell)$ is {\qsr} and of square free order, then we may write $g=g_1\ldots g_r$ for some pairwise commuting elements $g_i$ each of prime order. Since each $g_i$ is a power of $g$, it is itself \qsr. Then the above theorem shows $g_i \in \Sym(k)^\ell$ for each $i$ and so  $g\in \Sym(k)^\ell$. However, it is possible for {\qsres} to lie outside of $\Sym(k)^\ell$. Indeed, let $G=\Sym(5)\wr \Sym(2)$ and set $g=h\sigma$ where $h=((1,2),(3,4))$ and $\sigma =(1,2)\in \Sym(2)$. Then $g$ has order $4$ and $g^2=((1,2)(3,4),(1,2)(3,4))$ is {\qsr} with unique fixed point $(5,5)$. Hence $g$ is a {\qsre} and clearly $g\notin \Sym(5)^2$. Further, we see that no entry of $h$ is a {\qsre}.
\end{rem}

\begin{cor}
\label{cor: pa groups}
Suppose that $G$ is a primitive group of type PA preserving a product structure $\Delta^I$. Let $H\leqslant \Sym(\Delta)$ be such that $G \leqslant H \wr \Sym(I)$. Suppose that $G$ has a {\qsre}. Then $H$ contains {\qsres} and all {\qsres} of prime order in $G$ lie in the base group $H^I$.
\end{cor}
\begin{proof}
 If $G$ contains a {\qsre}, then $G$ contains a  {\qsre} $g$ of prime order. Since $G \leqslant \Sym(\Delta)\wr \Sym(I)$, it follows from Theorem~\ref{t:pa} that   $g\in \Sym(\Delta)^I$ and, writing $g=(h_1,\ldots,h_\ell)$, each $h_i$ is {\qsr}. Hence $g\in \Sym(\Delta)^I \cap G \leqslant H^I$ and $h_i \in H$ is a {\qsre} for all $i$.
\end{proof}

Our next result concerns the primitive group of diagonal type. We first establish some notation.

\begin{defn}
Let $k\geqslant 2$ be an integer and let $T$ be a finite nonabelian simple group. We define $N=T^k$ and $D=\{(t,\ldots,t) \mid t \in T\}$, the `straight diagonal' subgroup of $N$. We set $\Omega =[N:D]$ so that $|\Omega|=|T|^{k-1}$. The group $\Sym(k)$ naturally acts on $\Omega$  via
$$\sigma  : D(t_1,\ldots,t_k) \mapsto D(t_{1\sigma^{-1}},\ldots,t_{k\sigma^{-1}})$$
and further, for  $\alpha \in \Aut(T)$, we have
$$\alpha : D(t_1,\ldots,t_k) \mapsto D(t_1^\alpha,\ldots,t_k^\alpha).$$
A `maximal'  SD group is $W = T^k.(\Out(T) \times \Sym(k))$,  generated by $N$ together with all $\sigma\in\Sym(k)$ and $\alpha\in \Aut(T)$,    and is a  primitive permutation group on $\Omega$. A  group  $G$ with  $N \leqslant G \leqslant W$  such that $G$ is  primitive  on the simple direct factors of $N$ is  a primitive  group   of SD  type. \end{defn}

\begin{thm}
\label{thm: sd groups}
Let  $W=T^k.(\Out(T)\times \Sym(k))$ be a maximal primitive group of  type SD, as above. 
Then $W$ contains a {\qsre} if and only if $k$ is prime and $|T|$ is coprime to $k$. Moreover, the {\qsres} in $W$ are precisely those that are conjugate to elements of $\Sym(k)$ of order $k$.
\end{thm}
\begin{proof}
Let $\alpha$  be the coset $D$ in $N$ and let $H=W_\alpha$ denote a point-stabiliser in $W$. Then $H= \Aut(T) \times \Sym(k)$, where $k$ is the integer such that $\soc(W)=T^k$ and $T$ is a finite non-abelian simple group. 
Since $W=H\soc(W)$, for any $g\in W$ we have $Hg=H(t_1,\ldots,t_k)$ for some $t_i\in T$. Moreover, for any $(\alpha,\sigma)\in H$ we  have $H(t_1,\ldots,t_k)(\alpha,\sigma)=H(t_{1\sigma^{-1}}^\alpha,\ldots,t_{k\sigma^{-1}}^\alpha)$.

Suppose that $W$ contains a {\qsre}. Then we may assume there is $x\in W$ that is {\qsr} and  has prime order $p$. Now $|\Omega|=|T|^{k-1}$, and since $x$ is {\qsr}, we have $|\Omega|\equiv 1\pmod{p}$. Thus 
\begin{equation}
    \label{eq: T order coprime to p}
    |T| \text{ is coprime to  }p.
\end{equation}
Now $x$ fixes a unique $H$-coset, and since $\soc(W)$ acts transitively we may replace $x$ by a $\soc(W)$-conjugate if necessary, so that $x\in H$.
Write $x = ( \alpha, \sigma)$ where $\alpha\in\Aut(T)$
and $\sigma \in \Sym(I)$, where $I=\{1,\ldots,k\}$. Then $\sigma^p=1$ and $\alpha^p=1$.  We claim there exists $t\in T\setminus\{1\}$ such that $t^\alpha =t$. Indeed, if $\alpha=1$ then any non-trivial element suffices, and if $\alpha\neq 1$, then $|\alpha|=p$ and Thompson's Theorem \cite[Theorem 1]{Thompson} supplies such a $t$.  Suppose that $\sigma$ stabilises a proper subset $I' \subset I$. Set $t_i = t$ for $i\in I'$ and $t_j = 1$ for $j\notin I'$. Then
$$
H(t_1,t_2,\ldots,t_k) (\alpha,\sigma) = H(t_1,\ldots,t_k) \neq H
$$
and so $x$ is not {\qsr}. Thus no proper subset of $I$ is fixed by $\sigma$. Since $k>1$ and $\sigma^p=1$, we have  shown that  
\begin{equation}
    \label{eq: k=p}
    k=p \quad \text{and} \quad \sigma \neq 1.
\end{equation}
Together,  \eqref{eq: T order coprime to p} and \eqref{eq: k=p}   prove the forward direction of the first part  of  the  theorem.

Conversely, suppose that $k$ is prime,  that $|T|$ is coprime to $k$,  and consider the element $x=(1,\sigma)$ of $H$ of prime order $k$. We may assume that $\sigma$ is the $k$-cycle $(1,\ldots,k)$. Let $t_1,\ldots,t_k \in T$ be such that
$$H(t_1,\ldots,t_k) = H(t_1,\ldots,t_k)x = H(t_k,t_1,\ldots,t_{k-1}).$$
Then  $(t_1 t_k^{-1},t_2t_1^{-1},t_3t_2^{-1},\ldots,t_kt_{k-1}^{-1})\in H\cap \soc(G)$ and hence is equal to $(s,\ldots,s)$ for some $s\in T$. This implies that $t_2 =st_1$, that $t_3 = st_2 = s^2 t_1$, etc. In particular, we obtain $t_k = s^{k-1}t_1$ and so  $t_1 t_k^{-1}=s$ yields $s^k =1$. Since   $T$ has no elements of order $k$, we have $s=1$ and hence $t_1=t_2 =\cdots = t_k$ and so the only coset fixed by $x$ is $H$. Thus $x$ is {\qsr}.

Finally, we prove the moreover part of the theorem. We assume therefore that $k$ is prime and $|T|$ is coprime to $k$. Now let $x=(\alpha,\sigma)$ be a {\qsre} of $H$ of arbitrary order. Since every non-identity power of $x$ is a {\qsre}, the first part of the theorem shows   that $x$ is a $k$-element, say $|x|=k^a$. If $a>1$, the element  $x^{|x|/k} =(\alpha^{|x|/k},1)$ is {\qsr}, a contradiction to \eqref{eq: k=p}. Hence $x$ has prime order $k$. 
After conjugation by an element of  $H$, we may assume that $\sigma$ induces the $p$-cycle $(1,\ldots,p)$ on $I$. If $\alpha\neq 1$, then there is an element $t\in T\setminus\{1\}$ such that $t^\alpha \neq t$. Set $t_i = t^{\alpha^{i-1}}$. Then the coset $H(t_1,\ldots,t_p)$ is fixed by $(\alpha,\sigma)$ and is distinct from the coset $H$. Hence  $\alpha=1$ and $x=(1,\sigma)$. 
This proves that every  {\qsre} in $W$ is conjugate   to an element of  $\Sym(k)$ of order  $k$.
\end{proof}

\begin{rem}\label{rem:noSDeg}
 We note that not every primitive group of type SD with $k=p$ a prime and $|T|$ coprime to $p$ contains a \qsre. For example, if $T$ has an automorphism $\alpha$ of order $p$ then let $G=\langle T^p, (\alpha,(1,2,\ldots, p))\rangle\leqslant W$. Then a Sylow $p$-subgroup of $G$ is conjugate to $\langle (\alpha,(1,2,\ldots, p))\rangle$ and is not conjugate under an element of $W$ to an element of $\Sym(p)$. Thus by Theorem \ref{thm: sd groups}, $G$ does not contain a \qsre. One such instance of a suitable $T$ and $p$ is $T=\PSL_2(2^5)$ and $p=5$. 
\end{rem}

\begin{cor}\label{cor:CD}
Primitive groups of type CD with socle $N=T^m$ and $N_\alpha = T^\ell$, where $m=k \ell$ for some $k,\ell\geqslant2$,  have a {\qsre}   only if $k$ is prime and   $|T|$ is coprime to $k$.
\end{cor}
\begin{proof}
Let $G$ be as in the statement and suppose that $x\in G$ is a {\qsre}. We may assume that $|x|$ is prime. By definition we have $G \leqslant H \wr \Sym(\ell)$, where $H$ is a primitive group of SD type, so that $H \leqslant \Aut(T) \wr \Sym(k)$. By Corollary~\ref{cor: pa groups},  $x\in H^{\ell}$ and $H$ contains a {\qsre}. By Theorem~\ref{thm: sd groups},  $k=m/\ell$ is prime and $|T|$ is coprime to $k$. 
\end{proof}

\begin{rem}
\label{rem:CDeg}
 Note that a primitive group of type CD satisfying the conditions in Corollary \ref{cor:CD} need not contain a \qsre. For example, take $H$ to be a group provided by Remark \ref{rem:noSDeg} with no \qsres, then Corollary \ref{cor: pa groups} implies that $H\wr \Sym(\ell)$ does not contain a \qsre.

It is also possible to construct primitive groups of type CD that do not contain \qsres\  from groups of type SD that do contain a \qsre. For example, let $T$ be a simple group of order coprime to 5 (such as $\PSL_3(2)$), $k=5$ and $\ell=3$. Let $H=T^5\rtimes \langle \sigma\rangle$, where $\sigma=(1,2,3,4,5)$, be a primitive group of type SD acting on $\Delta$. Note that $\sigma$ is \qsr. Now let $G= \langle (T^5)^3, (\sigma,\sigma,1), (1,\sigma,\sigma)\rangle \rtimes \Sym(3)\leqslant H\wr \Sym(3)$ act on $\Delta^3$.
 By Corollary \ref{cor: pa groups}, any \qsre\ in $G$ lies in the base group $$B:=G\cap H^3=\langle (T^5)^3, (\sigma,\sigma,1),(1,\sigma,\sigma)\rangle$$ and is a product of \qsres. However, by Theorem \ref{thm: sd groups}, a \qsre\ in $\langle T^5,\sigma\rangle$ is conjugate to power of $\sigma$, and so a \qsre\ in $B$ must induce a semiregular permutation on the 15 simple direct factors of $\soc(G)$. However, $B$ contains no such element.
\end{rem}

\begin{thm}
\label{thm: tw type}
Suppose that $G$ is a primitive group of TW type. Then $G$ contains no {{\qsres}}.
\end{thm}
\begin{proof}
Let $T$ be the non-abelian simple group such that $\soc(G)=T^k$ and recall that the set acted upon is $T^k$ which is viewed as functions $f : \{1,\ldots,k\} \rightarrow T$. For each $i=1,\ldots, k$ we set 
$$T_i = \{ f \in T^k : f(j) = 1\text{ for all }j\neq i \}$$
so that $\soc(G)=T_1 \ldots T_k$. Let $H$ denote the stabiliser of the identity, so that $G=\soc(G) \rtimes H$ and let $x\in H$ have prime order $p$. Let $L$ denote the ``twisting subgroup'' of $H$, so that there is a homomorphism $\varphi : L \rightarrow \Aut(T)$, $k=|H:L|$ and $L=N_H(T_1)$.  If $x\in L$, then $\varphi(x)$ has order $1$ or $p$. By Thompson's Theorem  \cite[Theorem 1]{Thompson}, there is some $t\in T \setminus\{1\}$ such that $t^{\varphi(x)} = t$. Then the element $f\in T_1$ defined by $f(1) = t$ and $f(i)=1$ for $i>1$ satisfies $f^x = f$ and clearly $f\neq 1$. It follows that, if $x$ normalises any of the subgroups $T_i$, then $x$ is not a {\qsre}. Now suppose that $x$ induces a permutation of order $p$ on the set $\{T_1,\ldots,T_k\}$ with no fixed points. In particular $x\not\in L$, and without loss of generality, we assume that $x$ induces the $p$-cycle $T_1 \mapsto T_2 \mapsto  \cdots \mapsto T_p$ (and possibly some other $p$-cycles also). Pick $1 \neq f \in T_1$ 
and set $\hat{f} = f f^x \ldots f^{x^{p-1}} \in T_1\ldots T_p$. Since $f^{x^i} \in T_{i+1}$  (reading subscripts modulo $p$) and for $i\neq j$ we have $[T_i,T_j]=1$, we have  
$$
(\hat{f})^x = f^xf^{x^2} \ldots f^{x^{p-1}}f^{x^p}= f^x f^{x^2} \ldots f^{x^{p-1}} f = f f^x \ldots f^{x^{p-1}} = \hat{f}
$$
and since $f\neq 1$, the function $\hat{f}\neq 1$. Thus again $x$ is not a {\qsre}. It follows that $H$, and hence also $G$, contains no {\qsres}.
\end{proof}

\begin{thm}
\label{thm: hs groups}
Let $G$ be a primitive group of HS or HC type. Then $G$ contains no {{\qsres}}.
\end{thm}
\begin{proof}
First suppose that $G$ is a group of HS type. Note that the set acted upon is $T$, where $\soc(G)=T^2$ for some non-abelian simple group $T$. Let $G_1$ denote a point-stabiliser in $G$. Then $\Inn(T) \leqslant G_1 \leqslant \Aut(T)$. In particular, by Thompson's Theorem \cite[Theorem 1]{Thompson}, each element of $G_1$ of prime order fixes some non-trivial element of $T$, and hence fixes at least two points. Thus $G$ contains no {{\qsre}}.

Now suppose that $G$ is a group of HC type. Then $G$ is constructed from the product action of a HS type group $H$. Since $H$ has no {\qsre} by the previous paragraph, Corollary~\ref{cor: pa groups} implies that $G$ has no {\qsre}.
\end{proof}

Since primitive groups with regular normal subgroups have type HA, HS, HC or TW, 
from Theorems~\ref{thm: tw type} and~\ref{thm: hs groups} we obtain the following corollary.

\begin{cor}
    \label{cor:affine}
    Suppose that $G$ is a primitive group with a regular normal subgroup. If $G$ contains a {\qsre}, then $G$ must be of affine type.
\end{cor}

We also draw one further consequence, for {\qsres} of order two. It follows immediately from the results of this section, and the fact that a finite non-abelian simple group has even order.

\begin{cor}\label{cor:qsr invols}
    Suppose that $G$ is a primitive group containing a {\qsre} of order two. Then $G$ is of one of the  following types: HA, AS, PA.
\end{cor}

\section{The symmetric and alternating groups}

In this section, let $G$ be a primitive permutation group such that $\soc(G)=\Alt(n)$. Denote by $H$ a point-stabiliser in $G$ so that $\Omega = [G:H]$.  Set $[n]=\{1,\ldots,n\}$. 

We determine the cases where $G$ contains {\qsres}. We start with the exceptional cases related to $n=6$.

\begin{prop}
Suppose that $G$ has socle $\Alt(6)$ and $G\neq \Alt(6)$ or $\Sym(6)$. Let $H$ be a maximal subgroup of $G$. Then $G$ has a {\qsre} of prime order $p$ in the action on $[G:H]$ if and only if $|G:H|$, $H$ and $p$ appear in Table~\ref{tab:auta6}. In particular, any {\qsre} in $G$ lies in $\soc(G)$.
\end{prop}
\begin{proof}
This result is verified by a {\sc Magma} calculation.
\end{proof}

\begin{table}
    \centering
    \begin{tabular}{c|c|c}
         $|G:H|$ & $H$  & $p$   \\ \hline
          $36$ & $(\Sym(2)\times F_{20})\cap G$ & $5$ \\
          $10$ & $((\Sym(3)\wr\Sym(2)).2) \cap G$ & $3$ 
    \end{tabular}
    \caption{Actions of almost simple groups with socle $\Alt(6)$ that admit {\qsres} of order $p$.}
    \label{tab:auta6}
\end{table}

From now on, we may assume that $G=\Alt(n)$ or $G=\Sym(n)$.

\begin{prop}
\label{prop:k-set stab}
    Suppose that $\Omega$ is the set of $k$-subsets of $[n]$ for some integer $k$ such that $1\leqslant k < n/2$. Then $\Sym(n)$ contains a {\qsre} of prime order $p$ if and only if $k<p$ and $p$ divides $n-k$. Moreover, such a {\qsre} of order $p$ has cycle type $1^kp^{(n-k)/p}$.
\end{prop}
\begin{proof}
    Suppose $H$ fixes $\Delta$ with $|\Delta|=k$ and let $\Delta' = [n] \setminus \Delta$, and note that $|\Delta|<|\Delta'|$. Let $g\in H$ have order $p$ and suppose that $g$ is {\qsr}. Then $g$ leaves $\Delta$ and $\Delta'$ invariant. We claim there does not exist $A \subseteq \Delta$ and $B \subseteq   \Delta'$ such that $|A|=|B|$ and  $A^g =A$ and $B^g=B$. Indeed, for such subsets, $(\Delta \setminus A) \cup B$ is fixed by $g$, has size $k$, and is distinct from $\Delta$, so $g$ would not be {\qsr}, a contradiction.
    This implies in particular that $|\Delta|<p$, since otherwise we could choose $\langle g\rangle$-invariant subsets $A$, $B$ as above (each consisting of either $p$ fixed points of $g$, or a single $g$-cycle).  Thus $g$ fixes $\Delta$ pointwise. If $g$ fixed a point $x\in\Delta'$, then we could take $B=\{x\}$ and $A$ to be any singleton subset of $\Delta$, again yielding a contradiction. Thus $\Delta'$ is a union of $\langle g\rangle$-orbits of length $p$, so $p$ divides $n-k$ and  $g$ has cycle type $1^kp^{(n-k)/p}$.

    Conversely, suppose that there is a prime $p$ dividing $n-k$ and $p>k$. Write $n-k = pm$ and let $g\in \Sym(n)$ be of cycle type $1^kp^m$. A subset of $[n]$ is fixed by $g$ if and only if it is a union of the supports of the cycles of $g$.  Since $k< p$, for a $k$-set $\Delta$ to be fixed by $g$ it must consist of fixed points of $g$; and since $g$ has exactly $k$ fixed points, $\Delta = \mathrm{fix}(g)$. Thus $g$ fixes a unique $k$-subset, and so $g$ is {\qsr}.
\end{proof}

Since every element of odd prime order in $\Sym(n)$ lies in $\Alt(n)$, the above proposition gives necessary and sufficient conditions for $\Alt(n)$ to have a {\qsre} of odd prime order in its action on $k$-subsets of $[n]$. On the other hand, for involutions, we  have the following.

\begin{cor}
        Suppose that $\Omega$ is the set of $k$-subsets of $[n]$ for some integer $1\leqslant k < n/2$. Then $\Alt(n)$ acting on $\Omega$ contains a {\qsre} of   order $2$ if and only if $k=1$ and $n\equiv 1 \pmod{4}$. 
\end{cor}

\begin{prop}
\label{prop:k-part stab}
    Suppose that $\Omega$ is the set of partitions of $[n]$ into   parts of size $k$, for some divisor $k$ of $n$ with $1 < k \leqslant n/2$ and $n > 4$. Then $G\in\{\Sym(n), \Alt(n)\}$ contains a {\qsre} of prime order $p$ if and only if  $p$ is odd, and $n=km$ with  $k=p$ and  $2\leqslant m\leqslant p$. Moreover, if $g$ is a {\qsre}  of prime order $p$ and $\Gamma = \{ \Gamma_1,\ldots, \Gamma_m \}$ is the unique $g$-invariant partition on $[n]$ with parts of size $p$, then  $g\in\Alt(n)$, $g$ fixes each $\Gamma_i$ setwise, and either 
    \begin{enumerate}
        \item[(1)] $m=p$ and $g$ lies in the unique $G$-conjugacy class of elements with cycle type $1^pp^{m-1}$; or 
        \item[(2)] $2\leqslant m<p$, and $g$ lies in one of two $G$-conjugacy classes, namely elements with cycle type $1^pp^{m-1}$ or $p^{m}$.
    \end{enumerate}
\end{prop}

\begin{proof}
Let  $n=km$ and $G\in\{\Sym(n), \Alt(n)\}$, so that a partition stabiliser $H = (\Sym(k) \wr \Sym(m)) \cap G$. 
Suppose that $H$ fixes the partition
$$\Gamma = \{ \Gamma_1,\ldots, \Gamma_m \}$$
and let $\pi :H \rightarrow \Sym(\Gamma)$ be the action of $H$ on $\Gamma$. Suppose that $g\in H$ is a {\qsre} of prime order $p$. Assume first  that $\pi(g) \neq 1$. Then after relabelling the parts of $\Gamma$, we may assume that $\Gamma_1$, \ldots, $\Gamma_p$ form an orbit of $\langle \pi(g)\rangle$  and $\pi(g) : \Gamma_j \mapsto \Gamma_{j+1}$ (reading subscripts modulo $p$). Label the points in $\Gamma_1$ by $c_{1,i}$ for $1\leqslant i \leqslant k$, and then label the points of $\Gamma_2\cup \cdots \cup \Gamma_p$ as follows:
\[
c_{j,i} := (c_{1,i})^{g^{j-1}} \quad \text{for} \quad 2 \leqslant j \leqslant p, \ 1\leqslant i\leqslant k
\]
  (note that $c_{j,i} \in (\Gamma_1)^{g^{j-1}}=\Gamma_j$). Thus we have labelled all the points in $\Gamma_1,\ldots,\Gamma_p$ and $g$ effects the permutation $c_{j,i} \mapsto c_{j+1,i}$ (reading the subscripts $j, j+1$ modulo $p$). For $j=1,\ldots,p$ set  $\Gamma'_j = \{ c_{j,i} , c_{j+1,k} \mid 1 \leqslant i \leqslant k-1 \}$. Then $(\Gamma'_j)^g = \Gamma'_{j+1}$ (again, reading subscripts modulo $p$). Hence $g$ preserves the partition $$\Gamma ' := \left(\Gamma \setminus \{\Gamma_1,\ldots,\Gamma_p\} \right)\cup \{\Gamma'_1,\ldots,\Gamma'_p\}$$
and since $\Gamma\neq \Gamma'$, this contradicts $g$ being {\qsr}. Hence $\pi(g)=1$. 

\medskip
\noindent \underline{Claim}:  If  there exist nonempty subsets $A\subset \Gamma_i$, $B \subset  \Gamma_j$ such that $A^g = A$ and $B^g=B$ and $|A|=|B| < k$, then $i=j$.

Indeed, suppose $i,j,A,B$ are as in the statement and suppose that $i\ne j$. Define $\Gamma_i' = (\Gamma_i\setminus A) \cup B$ and $\Gamma_j'= (\Gamma_j \setminus B) \cup A$ and $\Gamma' = \{ \Gamma_i',\Gamma_j'\}\cup  \{\Gamma_k \mid k\neq i, j\}$. Then $g$ preserves each $\Gamma_k$ for $k\neq i, j$, and also $\Gamma_i'$ and $\Gamma_j'$. Hence $g$ fixes $\Gamma'$. Also  $\Gamma'\ne \Gamma$ since  $1\leqslant|A|=|B|<k$ implies that $|\Gamma_i\cap \Gamma_i'|\geqslant1$ and $|\Gamma_i \cap \Gamma_j'|\geqslant1$. This is a contradiction since $g$ is {\qsr}.  Hence $i=j$, and the Claim is proved.

\medskip
If $g$ has a fixed point in $\Gamma_1$ then, for each $j\ne 1$, it follows from the Claim that the $g$-action on $\Gamma_j$ is fixed point free.  Thus the $g$-action is fixed point free on at least one of $\Gamma_1$ and $\Gamma_2$, and we may assume that $g$ acts fixed point freely on $\Gamma_2$.     In particular $p$ divides $|\Gamma_2|=k$.   If $k\geqslant 2p$ then choosing a $g$-cycle $B\subset \Gamma_2$ of length $p$, and a $g$-invariant $p$-subset $A\subset \Gamma_1$ (either a $g$-cycle of length $p$ or a set of $p$ fixed points of $g$), we obtain a contradiction to the Claim. Hence $k=p$. Further, this argument shows that either $g$ is fixed point free on $[n]$ with cycle type $p^m$, or  $g$ fixes exactly one part pointwise and has cycle type $1^pp^{m-1}$.

Suppose that $g$ has at least $p$ cycles of length $p$, so that $g$ acts non-trivially on (at least) $p$ parts, say  $\Gamma_1$, \ldots, $\Gamma_p$. Label the points in these parts as $\Gamma_j = \{c_{j,i}\mid 1\leqslant i\leqslant p\}$, for $j=1,\dots, p$, such that $g : c_{j,i} \mapsto c_{j,i+1}$, reading the subscripts $i, i+1$ modulo $p$. For $1\leqslant i \leqslant p$, set $\Gamma_i'=\{c_{j,i} \mid 1\leqslant j \leqslant p\}$. Then the partition $\Gamma':=\{ \Gamma_1',\dots,\Gamma_p',\Gamma_{p+1},\dots, \Gamma_m\}$  is different from $\Gamma$ and is fixed by $g$, a contradiction to $g$ being {\qsr}. We conclude that either $g$ has cycle type $p^m$ with $m\leqslant p-1$, or  $g$ has cycle type $1^pp^{m-1}$ with $m\leqslant p$. In either case, since $n>4$ and $m\geqslant2$, it follows that $p$ is odd and hence that $g\in\Alt(n)$.  Note that both $\Alt(n)$ and $\Sym(n)$ have a unique conjugacy class of elements of this cycle type -- this is clear for $\Sym(n)$ and for $\Alt(n)$ we see that $C_{\Sym(n)}(x)$ is not contained in $\Alt(n)$ for $x$ of cycle type $1^pp^{m-1}$ or $p^m$, and so $\Alt(n)$ acts transitively on the $\Sym(n)$-conjugacy class of $x$.

It remains for us to prove that elements with these cycle types are {\qsr}. Suppose first that $m \leqslant p$ and that $g \in \Alt(n)$ has cycle type $1^pp^{m-1}$ ($p$ fixed points and $m-1$ cycles of length $p$). Let $A_1,\ldots,A_{m-1}$   be the non-trivial orbits of $\langle g\rangle $ and let $F=\mathrm{fix}(g)$. Then $|F|=p$ and  $g$ preserves the partition 
$$
\Gamma=\{A_1,\ldots,A_{m-1}, F\}.
$$
Suppose that $\Gamma'=\{\Gamma_1,\ldots,\Gamma_m\}$ is a partition with parts of size $p$ fixed by $g$, and without loss of generality suppose that $\Gamma_1$ contains a point of $F$ (fixed by $g$). Then $g$ fixes $\Gamma_1$ setwise, and since $|\Gamma_1|=p$ it follows that $g$ must fix $\Gamma_1$ pointwise, so $\Gamma_1=F$. Also $g$ permutes among themselves the remaining $m-1$ parts of $\Gamma'$, and since $m-1<p$ this implies that $g$ fixes each $\Gamma_i$ setwise. Thus, for each $i>1$, $\Gamma_i$ is an orbit of $\langle g\rangle$, and it follows that $\Gamma'=\Gamma$. Hence $g$ is {\qsr}.

A similar, but easier argument shows that, if $m\leqslant p-1$ and  $g \in \Alt(n)$ has cycle type $p^{m}$, then $g$ fixes a unique partition of $[n]$ with $m$ parts of size $p$, namely the partition with parts the $\langle g\rangle$-orbits.  This completes the proof.
\end{proof}

\begin{thm}
\label{thm:alt or sym}
Let $G = \Alt(n)$ or $\Sym(n)$ and let $H$ be a maximal subgroup of $G$ not containing $\Alt(n)$. In the action of $G$ on $[G:H]$, $G$ contains a {\qsre} of prime order $p$ if and only if one of the following holds:
\begin{enumerate}[(1)]
\item $H$ and $p$ are as in one of the lines of Table~\ref{tab:infinite families}; or
\item $G=\Alt(n)$ and $H$, $p$ and $n$ are as in one of the lines of Table~\ref{tab:exceps for altn}.
\end{enumerate}
\end{thm}

\begin{proof}
First note that the theorem is true for $n<5$, and easily confirmed by {\sc Magma}. We assume therefore that $n\geqslant 5$, and let $\Omega = [G:H]$. The case where  
$H$ is the stabiliser of a $k$-subset of $[n]$ with $1\leqslant k < n/2$, is handled by Proposition~\ref{prop:k-set stab}, and the case where $H$ is the stabiliser of a  partition of $[n]$, with parts of size $k$, is handled by Proposition~\ref{prop:k-part stab}. These cases correspond to part  (1)  of the theorem and lines $1$  and $2$ of Table~\ref{tab:infinite families}. From now on we may assume that $H$ is a maximal subgroup of $G$ that is transitive and primitive on $[n]$.

Suppose that $x\in H$ has prime order $p$ and is {\qsr} in the action of $G$ on $[G:H]$. Lemma~\ref{lem: qs and sylow} implies that $H$ contains a Sylow $p$-subgroup $S$ of $G$. Hence $H$ contains an element $y$ of order $p$ such that $y$ is a $p$-cycle on $[n]$. If $p \leqslant n-3$, then a theorem of Jordan \cite[Theorem 3.3E]{DixonMortimer} implies that $H$ contains $\Alt(n)$, a contradiction. Hence $n-2 \leqslant p \leqslant n$. In particular, since $n\geqslant 5$, the prime $p$ is odd and $p>n/2$, so $p^2$ does not divide $n$ and $S:=\langle x \rangle$ is a Sylow $p$-subgroup of both $G$ and $H$. Note that  $x$ has $k$ fixed points on $[n]$ where $0 \leqslant k \leqslant 2$. If $k=2$, then  $N_{\Sym(n)}(S)$ contains a transposition and Lemma~\ref{lem: qs and sylow} shows that $N_G(S) =  N_H(S) $; since $H$ does not contain a transposition, we must have $G = \Alt(n)$.

The study of primitive permutation groups containing such a  $p$-cycle has a rich history. Such groups are well understood, and a list of them  can be found in \cite[Theorem 1.2]{Jones}. We consider the outcomes of that theorem below.   We note that, by Corollary~\ref{cor: normalisers equals implies qsr}, an element $x\in H$ of order $p$ is {\qsr} if and only if $N_H(\langle x \rangle)=N_G(\langle x \rangle)$ (since $\langle x \rangle$ is a Sylow $p$-subgroup of $G$). 

If \cite[Theorem 1.2(1)(a)]{Jones} holds, then $n=p$ and $H= \AGL(1,p) \cap G$.   Also $N_H(\langle x \rangle)=N_G(\langle x \rangle)$ for an element $x$ of order $p$, so $G$ contains a {\qsre}.  Moreover, the subgroup $H$ is indeed maximal in $G$, except when  $G = \Alt(p)$ and $p=7,11,17,23$, see \cite[Theorem(I)]{LPSAltSymMax}. This gives row $3$ of Table~\ref{tab:infinite families} and we are in case (1) of our theorem.

If \cite[Theorem 1.2(1)(b)]{Jones} holds, then  $n=p=(q^d-1)/(q-1)\geqslant5$ and $\PGL(d,q) \leqslant H \leqslant \mathrm{P}\Gamma\mathrm{L}(d,q)$. In particular, $d$ must be prime. Writing $q=r^e$, we see that
  $(q^d-1)/(q-1).d=pd$ divides $|N_H(S)|$, and $|N_H(S)|$ divides 
$(q^d-1)/(q-1).de=pde.$
On the other hand $|N_G(S)| = p(p-1)/t$ (where $t=1,2$ for $G=\Sym(p)$,  $\Alt(p)$, respectively). By  Lemma~\ref{lem: qs and sylow} we have $N_G(S) = N_H(S)$. 
Hence $p(p-1)/t = pde'$, for some divisor $e'$ of $e$, so $tde'=p-1 = q(q^{d-1}-1)/(q-1)= q(q^{d-2}+\cdots+q+1)$. If $r$ were odd then this would imply that $q=r^e$ divides $de$ (since $t\leqslant2$ and $e'\vert e$), which in turn implies that $d=r$ (since $d$ is prime) and $e=e'=1$. However we would then have $2r\geqslant tde'\geqslant r(r^{r-2}+1)$, which is not possible. Hence $r=2$.  
If $d=2$ then we require $q=2^e\geqslant8$ (since $H\not\geqslant \Alt(n)$), and $2te'=tde'=q=2^e$; this implies that $e=e'=4$ so  $H=\mathrm{P}{\Gamma}\mathrm{L}(2,2^4)$, which is a maximal subgroup of  
$\Alt(17)$ by  \cite[Theorem(II)]{LPSAltSymMax}.   Here, for an element $x\in H$ of order $17$, $N_H(\langle x \rangle)= \langle x \rangle.8= N_G(\langle x \rangle)$, so $x$ is {\qsr}. Thus we are in line 6 of Table~\ref{tab:exceps for altn}, and case (2) of our theorem.
So we may assume that $d\geqslant3$.
 Now $e'\leqslant e\leqslant 2^e=q$ and $t\leqslant 2$, so  
\[ 
q^{d-2}+\cdots + q+1 = \frac{tde'}{q} \leqslant 2d,
\]
and it follows that $d= 3$ (since $d$ is prime) and $q\leqslant 4$. Thus $3te'=2^e(2^e+1)$, and this implies that $q=t=2$, so $H=\PSL(3,2)$ which is a maximal subgroup of $G=\Alt(7)$,   and an element $x\in H$ of order $7$ is {\qsr} since $N_H(\langle x \rangle)= \langle x \rangle.3= N_G(\langle x \rangle)$. Thus we are in line $1$ of Table~\ref{tab:exceps for altn} and in case (2) of our theorem.

If \cite[Theorem 1.2(1)(c)]{Jones} holds (again $n=p$), then $p=11$ and $H$ is $\PSL(2,11)$ or  $M_{11}$, or $p=23$ and  $H=M_{23}$. Since we require $H$ to be maximal, we have $H=M_{11}$ or $H=M_{23}$. In these cases, an element of $H$ of order $p$ is {\qsr} by Corollary~\ref{cor: normalisers equals implies qsr}. Thus we are in case (2) of the theorem and this gives lines $4$ and $7$ of Table~\ref{tab:exceps for altn}.

If \cite[Theorem 1.2(2)(a)]{Jones} holds, then $n=q^d>4$, $\AGL(d,q)\leqslant H\leqslant \mathrm A\Gamma \mathrm
 L(d,q)$, and $p=n-1=q^d-1$ is prime. The latter implies that $q=2^e$ and  either $d=1$ with $e$  prime, or $e=1$ with $d$ prime. Since $H$ is maximal in $G$, and $\mathrm A\Gamma \mathrm
 L(1,2^e) \leqslant \AGL(e,2)$, it follows that $e=1$,   $H=\AGL(d,2)$, and  $ G = \Alt(n)$. Now for $x\in H$ of order $p=2^d-1$, the normalisers $N_H(\langle x \rangle) = C_p.C_d$ and $N_G(\la x \ra) = C_p.C_{(p-1)/2}$, so $x$ is {\qsr} if and only if these subgroups are equal, that is, $(p-1)/2 = d$, or equivalently, $2d+1=p=2^d-1$. Since $n=2^d>4$, this implies that $d=3$ and $H =\AGL(3,2)$. This gives case (2) and the entry in line $2$ of Table~\ref{tab:exceps for altn}.

If \cite[Theorem 1.2(2)(b)]{Jones} holds, then $n=p+1$ and $H=\PSL(2,p)$ or $\PGL(2,p)$ (note that $\PGL(2,p)\nleqslant \Alt(p+1)$ so $G=\Sym(n)$ in this case). Here, for $x\in H$ of order $p$,  $N_H(\langle x\rangle) = N_G(\langle x\rangle)$, so $x$ is {\qsr}. This gives case (1) and the entry in line 4 of Table~\ref{tab:infinite families}.

If \cite[Theorem 1.2(2)(c)]{Jones}  holds, then either $n=12$ and $H$ is $M_{11}$ or  $M_{12}$,  or $n=24$ and $H=M_{24}$. Since $H$ is maximal in $G$, the group $H$ is $M_{12}$ or $M_{24}$, and $G=\Alt(12)$ or $\Alt(24)$, respectively. This gives case (2) and the entries in lines  $5$ and $8$ of Table~\ref{tab:exceps for altn}.

If \cite[Theorem 1.2(3)]{Jones} holds, then $n=p+2$ (so $G=\Alt(n)$) and $\PGL(2,q) \leqslant H \leqslant \mathrm{P}\Gamma \mathrm{L}(2,q)$ with $n=q+1$. Thus $n>4$ (since $H\not\geqslant\Alt(n)$), and $p=n-2 = q-1$ is prime. It follows that $q=2^e$ and $p=2^e-1$ is a Mersenne prime, so in particular $e$ is an odd prime  and hence $\mathrm{P}\Gamma \mathrm{L}(2,q)\leqslant \Alt(n)=G$. Since $H$ is maximal, we have $H=\mathrm{P}\Gamma \mathrm{L}(2,q)$.  Now, for $x\in H$ of order $p$, we have $N_H(\langle x \rangle) = D_{2p} . C_e$ and $N_G(\langle x \rangle) = C_p.C_{p-1}$, so $x$ is {\qsr} if and only if these subgroups are equal, that is, $2e=p-1 = 2^e-2$. This implies, since $e$ is an odd prime, that $e=3$,   $n=9$ and $H=\mathrm{P}\Gamma\mathrm{L}(2,8)$. This gives case (2) and the entries in line  $3$ of Table~\ref{tab:exceps for altn}.
\end{proof}

\begin{rem} Concerning row $3$ of Table~\ref{tab:infinite families}. For $p = 7,11,17,23$ and $G=\Alt(p)$, the group $\AGL(1,p) \cap \Alt(p) \cong C_p.C_{(p-1)/2}$ is not maximal in $\Alt(p)$, and is contained in subgroups isomorphic to $\PSL(3,2)$, $M_{11}$, $\mathrm P \Gamma \mathrm L(2,16)$, $M_{23}$, respectively.

    Concerning row $4$ of Table~\ref{tab:infinite families}, when $p$ is odd and $n=p+1$, an easy calculation shows  $\PSL(2,p)=\PGL(2,p) \cap \Alt(n)$.
\end{rem}

\begin{cor}
    Suppose that $G$ is a primitive permutation group on a set $\Omega$ and that $\soc(G)=\Alt(n)$ for some $n\geqslant 5$. If $g\in G$ is a {\qsre} of order  $2$, then $n$ is odd and the $G$-action on $\Omega$ is its natural action on $n$ points. Moreover if $n\equiv 3 \pmod 4$, then $G=\Sym(n)$.
\end{cor}

\section{The sporadic groups}

In this final section we consider  primitive permutation groups $G$ with point-stabiliser $H$ such that $\soc(G)$ is one of the twenty-six sporadic simple groups. Most of the work can be done computationally using \cite{GAP4}, but the Monster requires specialised arguments.

\subsection{Computations}

\begin{lem}
\label{lem: sporadic but not monster}
    Let $G$ be an almost simple group with socle a sporadic simple group other than the Monster. Let $H$ be a maximal subgroup of $G$. Then there is a {\qsre} of prime order $p$ in the action of $G$ on $[G:H]$ if and only if there is a $G$-conjugacy class of elements of order $p$ in the row corresponding to $H$ in one of Tables~\ref{tab:spor1}--\ref{tab:spor3}.
\end{lem}
\begin{proof}
The proof is computational and makes use of the GAP \cite{GAP4} package CTblLib \cite{CTblLib}. For a group $G$ as in the statement, the character table of $G$ is stored in GAP. We illustrate the process for $G=M_{11}$ below:
\begin{verbatim}
    t:=CharacterTable("M11");;
\end{verbatim} If the list of maximal subgroups of $G$ is known, then the function `HasMaxes' returns true on the character table of $G$. In this case the list of maximal subgroups can be found and stored as follows:
\begin{verbatim}
    maxes:=Maxes(t);;
\end{verbatim}
and for the $i$th maximal subgroup, its character table and the fusion rules, the membership map from the set of conjugacy classes of $H$ to the set of conjugacy classes of $G$ (if known), are found using:
\begin{verbatim}
    ct:=CharacterTable(maxes[i]);;
    Fus:=GetFusionMap(ct,t);;
\end{verbatim}
With all of this information at hand, for a given element $x \in H$, we can compute $|C_G(x)|$ and $|C_H(x)|$ from the character tables (the function `SizesCentralisers' is implemented to do this). Further, using the fusion map we can compute whether $x^G \cap H = x^H$. We thus have enough information to decide if the element $x$ is {\qsr} or not. This procedure works for all sporadic simple groups (other than the Monster), except when $G=B$, the Baby Monster, and $H=(2^2 \times F_4(2)).2$, and in this case the fusion map is unknown. We now treat this exceptional case.

If $x\in H$ is a {\qsre} of prime order $p$ then $|G|_p=|H|_p$, by Lemma~\ref{lem: qs and sylow}(1),  and so $p\in \{7,13,17\}$. If $p=7$, then $5 \mid |C_G(x):C_H(x)|$ and if $p=13$ then $3 \mid |C_G(x):C_H(x)|$, and hence elements of these orders $p$ are not {\qsr}, by Lemma~\ref{lem: qs and sylow}(3). Thus $p=17$. Since $|C_G(x)|=2^2.17$, we see $C_G(x)=C_H(x)$. Further, $F_4(2)$ has two classes of elements of order $17$, which are fused in $F_4(2):2$ (since the centraliser does  not grow), and hence in $H$. Thus elements of $H$ of order $17$ are {\qsr}, by Corollary~\ref{cor: qsr iff cent condition}.
\end{proof}

\subsection{The Monster}

The maximal subgroups of the Monster have been well studied, and `almost complete' lists have existed since the publication of the Atlas \cite{ATLAS} in 1985. The `unknown' maximal subgroups were believed to belong to a short list \cite{Wilson}. Recently, the `mmgroup' package of Seysen \cite{Seysen1,Seysen2,Seysen3} caused a paradigm shift in the study of the Monster and its subgroups.  Dietrich, Lee and Popiel \cite{DLP} were the first to cast new light on the study of maximal subgroups using this package, and found several corrections to the known lists. As such, the maximal subgroups we consider below are drawn from \cite{DLP,DLPP}.

\begin{lem}
\label{lem: sporadic and monster}
    Let $G$ be the Monster sporadic simple group and let $H$ be a maximal subgroup of $G$. There is a {\qsre} of prime order $p$ in the action of $G$ on $[G:H]$ if and only if $p$ appears in the row corresponding to $H$ in Table~\ref{tab:M}.
\end{lem}

\begin{proof}
    For maximal subgroups    whose character table and fusion rules are stored in \cite{CTblLib}, we can compute the {\qsres} in GAP \cite{GAP4} using the procedure described in the proof of  Lemma~\ref{lem: sporadic but not monster}.  (The command \begin{verbatim}NamesOfFusionSources(CharacterTable("M"));\end{verbatim} tells us which of the maximal subgroups of the Monster we can apply  the described method to.)
    The results computed using GAP are indicated by the entry `GAP' in the `Notes' column of Table~\ref{tab:M}. 

    From now on we assume that $H$ is a group for which the character table is not stored in \cite{CTblLib}. We will consider each maximal subgroup $H$ in turn. We assume that $x\in H$ is a {\qsre} of prime order $p$. In particular, by Lemma~\ref{lem: qs and sylow}, 
    \begin{equation}
        \label{eq:gp=hp} |G|_p=|H|_p \quad \mbox{and}\quad C_G(x)=C_H(x).
    \end{equation} 
    Note that $G$ has two classes of involutions, the centralisers of which are maximal. In particular, if $H$ is a maximal subgroup that is not an involution centraliser, then no involution in $H$ is {\qsr}. This means (since the two involution centralisers could be, and were, considered using GAP) that we may assume $p>2$ from now on.

    Suppose that $H=2^{5+10+20}.(S_3 \times L_5(2))$. From \eqref{eq:gp=hp} we have that $p=31$. Now $G$ has two classes of elements of order $31$ and representatives may be chosen to be inverses of each other \cite{ATLAS}. Since $L_5(2)$ has six classes of such elements, each class must split into three $H$-conjugacy classes. It follows that $x$ is not {\qsr} by Corollary~\ref{cor: qsr iff cent condition}.

    Next suppose that $H=59:29$. From \eqref{eq:gp=hp} we have $p=29$ or $59$. In the first case  $H$ does not contain $C_G(x)$, whereas in the latter case, we have an example, as in the top right line of Table~\ref{tab:M}. 
    \end{proof}

\subsection{Proof of Theorem~\ref{thmintro: sporadic} and Corollary~\ref{cor: more than one class of qsrs}}

Theorem~\ref{thmintro: sporadic} follows from Lemma~\ref{lem: sporadic but not monster} and Lemma~\ref{lem: sporadic and monster}.

\medskip
For Corollary~\ref{cor: more than one class of qsrs}, $G$ is a primitive permutation group of degree $n$ with socle an alternating or sporadic group. We first consider the case where $\soc(G)$ is alternating. The result is vacuously true if $n<2p$ since $G$ has only one class of subgroups of order $p$; and we see that this holds for all rows of Table~\ref{tab:exceps for altn} and all but the first two rows of   Table~\ref{tab:infinite families}. For  rows $1$ and $2$ of Table~\ref{tab:infinite families} we use Propositions~\ref{prop:k-set stab} and ~\ref{prop:k-part stab} respectively, which show that a {\qsre} of prime order $p$ has a unique cycle type unless $n<p^2$.

Now we consider the case that $\soc(G)$ is a sporadic group and has a {\qsre} of prime order $p$.  If $|G|_p=p$, then it follows from Sylow's Theorem that there is a unique conjugacy class of {\qsr} subgroups. Further, if there is a unique class of {\qsres} of order $p$, then there is nothing to do. Considering Tables~\ref{tab:M}-\ref{tab:spor3}, we are left to consider the pairs $(\soc(G),p)=(HS,5)$, $(J_2,5)$, $(McL,5)$, $(He,7)$ and $(J_4,11)$. For $\soc(G)=J_2$ and $He$, we make use of the power maps \cite{ATLAS} and \cite{OnlineAtlas} which tell us that the {\qsres} of order $p$ are powers of each other, and therefore generate the same subgroup. The remaining cases lead to genuine examples; if $\soc(G)=McL$ for example, we can see that $5a$ and $5b$ elements have centralisers of different orders, and so generate non-conjugate subgroups of order five.

\subsection{Tables of Results}
\label{sec:tables}

\begin{table}[h]
    \centering
    \begin{tabular}{|c|c|c||c|c|c|} \hline
         $H$&  class & Notes & $H$& $p$ & Notes \\ \hline
         $2.B$                              &  $47a$  $47b$   &  GAP  & $59:29$                       & $59a$  $59b$  &   \\
         $2^{2+11+22}.(M_{24}\times S_3)$   &  $23a$  $23b$   &    GAP   & $(A_5 \times U_3(8):3):2$     & $19a$    &    GAP\\
         $3.Fi_{24}$                        &  $29a$          & GAP   & $(L_3(2) \times S_4(4):2).2$  & $17a$    &  GAP\\
         $S_3 \times Th$                    &  $31a$  $31b$   & GAP   & $L_2(71)$                     & $71a$  $71b$    &  GAP \\
         $41:40$                            &  $41a$          &  GAP  &                               &       &  \\ \hline
    \end{tabular}
    \caption{Maximal subgroups $H$ of the Monster group containing a {\qsre} belonging to the class given in the `class' column.}
    \label{tab:M}
\end{table}

\begin{table}[h]
\centering

\begin{tabular}{|l | l || l|l| }  \hline \underline{$G$}, $H$ & class & \underline{$G$}, $H$  & class \\
\hline \hline
\underline{$ON$}                &                           & \underline{$J_3$}               &  \\
$L_3(7).2$                      & $7b$ $19a$--$c$    & $L_2(16).2$                     & $5a$ $5b$  \\
$L_3(7).2$                          & $7b$ $19a$--$c$    & $L_2(19)$                       & $19a$ $19b$  \\
$J_1$                           & $11a$ $19a$--$c$   & $L_2(19)$                          & $19a$ $19b$  \\
$(3^2:4 \times A_6).2$                          & $5a$                      & $L_2(17)$                       & $17a$ $17b$  \\
$L_2(31)$                       & $31a$ $31b$               & $(3 \times A_6).2$                          & $5a$ $5b$  \\
$L_2(31)$                          & $31a$ $31b$               & $3^2.3^{1+2}:8$                 & $3b$  \\ \hline

\underline{$ON.2$}              &                           & \underline{$J_1$}             &  \\
$J_1\times 2 $                  & $11a$ $19a$--$c$   & $19:6$                        & $19a$--$c$  \\
$(3^2:4\times A_6).2^2$         & $5a$                      & $11:10$                       & $11a$  \\
$7^{1+2}_+:(3\times D_{16})$    & $7b$                      & $D_6\times D_{10}$            & $3a$ $5a$ $5b$   \\
$31:30$                         & $31a$                     & $7:6$                         & $7a$  \\ \hline

$\underline{J_3.2}$           &                           & \underline{$J_4$}             &  \\
$L_2(16).4$                   & $5a$                 & $2^{3+12}.(S_5 \times L_3(2))$                        & $7a$ $7b$ $5a$  \\
$ L_2(17)\times 2$            & $17a$ $17b$               & $11_+^{1+2}:(5\times 2S_4)$   & $11a$ $11b$  \\
$(3\times M_{10}):2$          & $5a$               & $L_2(32).5$                   & $31a$--$c$  \\
$3^2.3^{1+2}:8.2$             & $3b$               & $L_2(23).2$                   & $23a$  \\
$19:18$                          & $19a$                 & $29:28$                        & $29a$  \\
                                &                       & $43:14$                       & $43a$--$c$  \\
&& $37:12 $                      & $37a$--$c$  \\ \hline
\end{tabular}
    \caption{Almost simple groups $\underline{G}$  with socle a sporadic simple group with maximal subgroup $H$ such that $H$ contains a {\qsre} belonging to class indicated in the class column.}
    \label{tab:spor1}
\end{table}

\begin{table}
\centering

\begin{tabular}{ | l| l || l | l|| l| l|}   \hline    \underline{$G$}, $H$ & class & \underline{$G$}, $H$  & class& \underline{$G$}, $H$ & class \\
\hline \hline
\underline{$M_{11}$}        &                   & \underline{$M_{24}$}       &               & \underline{$Co_2$} &  \\
$A_6.2_3$                   & $5a$              & $M_{23}$                   & $23a$ $23b$   & $U_6(2).2$  & $11a$ \\
$L_2(11)$                   & $11a$ $11b$       & $M_{22}.2$                 & $11a$         & $2^{10}:M_{22}:2$  & $11a$ \\
$3^2:Q_8.2$                 & $3a$              & $M_{12}.2$                 & $11a$         & $2^{1+8}:S_6(2)$  & $7a$ \\
$A_5.2$                     & $5a$              & $2^6:3.S_6$                & $5a$          & $HS.2$  & $11a$ \\
                            &                   & $L_3(4).3.2_2$             & $7a$ $7b$     & $M_{23}$  & $23a$ $23b$ \\
\underline{$M_{12}$}        &                   & $2^6:(L_3(2)\times S_3)$   & $7a$ $7b$     & $5^{1+2}:4S_4$  & $5a$ \\
$M_{11}$                    & $11a$ $11b$       & $L_2(23)$                  & $23a$ $23b$   &  &  \\
$M_{11}$                    & $11a$ $11b$       &                            &               & \underline{$Co_3$} &  \\
$A_6.2^2$                   & $5a$              & \underline{$HS$}           &               & $McL.2$  & $5a$ $11a$ $11b$ \\
$A_6.2^2$                   & $5a$              & $M_{22}$                   & $11a$ $11b$   & $M_{23}$  & $23a$ $23b$ \\
$L_2(11)$                   & $11a$ $11b$       & $U_3(5).2$             & $5a$ $5c$ $7a$    & $3^5:(2\times M_{11})$  & $11a$ $11b$ \\
$2\times S_5$               & $5a$              & $U_3(5).2$                 & $5a$ $5c$ $7a$    & $U_3(5).3.2$  & $5a$ \\
                            &                   & $L_3(4).2_1$           & $7a$              & $L_3(4).D_{12}$  & $7a$ \\
\underline{$M_{12}.2$}      &                   & $A_8.2$                & $7a$              & $2\times M_{12}$  & $11a$ $11b$ \\
$L_2(11).2 $                & $11a$             &$ M_{11}$               & $11a$ $11b$       & $S_3\times L_2(8).3$  & $7a$ \\
$L_2(11).2 $                & $11a$             & $ M_{11}$                 & $11b$ $11a$       &  &  \\
$(2^2\times A_5):2$         & $5a$              &                        &                   & \underline{$McL$} &  \\
                            &                   & \underline{$HS.2$}     &                   & $M_{22}$  & $11a$ $11b$ \\
\underline{$M_{22}$}        &                   & $M_{22}.2$                & $11a$             & $M_{22}$  & $11a$ $11b$ \\
$L_3(4)$                    & $7a$ $7b$         &$ L_3(4).2^2$           & $7a$              &$ L_3(4).2_2 $ & $7a$ $7b$ \\
$A_7$                       & $5a$ $7a$ $7b$    & $S_8\times 2 $          & $7a$              &$ 2.A_8 $ & $7a$ $7b$ \\
$A_7$                       & $5a$ $7a$ $7b$    & $5^{1+2}:[2^5]$               & $5a$ $5c$         & $2^4:A_7$  & $7a$ $7b$ \\
$2^4:S_5$                   & $5a$              &                        &                   & $2^4:A_7$  & $7a$ $7b$ \\
$2^3:L_3(2)$                & $7a$ $7b$         & \underline{$J_2$}      &                   & $M_{11}$  & $11a$ $11b$ \\
$A_6.2_3 $                  & $5a$              & $3.A_6.2_2$            & $3a$              & $5^{1+2}:3:8$  & $5a$ $5b$ \\
$L_2(11)$                   & $11a$ $11b$       & $L_3(2).2$             & $7a$              &  &  \\
                            &                   & $5^2:D_{12}$           & $5c$ $5d$         & \underline{$McL.2$} &  \\
\underline{$M_{22}.2 $}     &                   &                        &                   & $L_3(4).2^2$  & $7a$ \\
$L_3(4).2_2$                & $7a$ $7b$         & \underline{$J_2.2$}    &                   & $ 2.A_8.2 $  & $7a$ \\
$2^5.S_5$                   & $5a$              & $3.A_6.2^2$            & $3a$              & $2\times M_{11}$  & $11a$ $11b$ \\
$2\times 2^3:L_3(2)$        & $7a$ $7b$         & $L_3(2).2\times 2$     & $7a$              & $5^{1+2}:(24:2)$  & $5a$ $5b$ \\
$A_6.2^2 $                  & $5a$              & $5^2:(4\times S_3) $   & $5b$              &  &  \\
$L_2(11).2$                 & $11a$             &                        &                   & \underline{$Suz$} &  \\
                            &                   & \underline{$Co_1$}     &                   & $G_2(4)$  & $13a$ $13b$ \\
\underline{$M_{23}$}        &                   & $Co_2$                 & $23a$ $23b$       & $J_2.2$  & $5b$ \\
$M_{22}$                    & $11a$ $11b$       & $3.Suz.2 $             & $11a$             & $(A_4\times L_3(4)):2$  & $7a$ \\
$L_3(4).2_2$                & $7a$ $7b$         & $2^{11}:M_{24} $       & $23a$ $23b$       & $M_{12}.2 $ & $11a$ \\
$2^4:A_7$                   & $7a$ $7b$         & $Co_3 $                & $23a$ $23b$       & $L_3(3).2 $ & $13a$ $13b$ \\
$A_8$                       & $5a$              & $(A_4\times G_2(4)):2 $ & $13a$            & $L_3(3).2 $ & $13a$ $13b$ \\
$M_{11}$                    & $11a$ $11b$       &                        &                   &  &  \\
$2^4:(3\times A_5).2$       & $5a$              &                        &                   &  &  \\
$23:11$                     & $23a$ $23b$       &                        &                   &  &  \\
\hline
\end{tabular}
    \caption{Almost simple groups $\underline{G}$  with socle a sporadic simple group with maximal subgroups $H$ such that $H$ contains a {\qsre} belonging to class indicated in the class column.}
    \label{tab:spor2}
\end{table}

\begin{table}
\centering

\begin{tabular}{|l | l || l | l|} \hline  \underline{$G$}, $H$ & class & \underline{$G$}, $H$  & class \\
\hline \hline
\underline{$Suz.2$}                     &                   &       \underline{$Fi_{22}.2$}                    &  \\
$G_2(4).2$                   & $13a$             & $2.U_6(2).2$                       & $11a$ \\
$J_2.2\times 2$              & $5b$              & $O_8^+(2):S_3\times 2$               & $5a$ $7a$ \\
$(A_4\times L_3(4):2_3):2$  & $7a$               & $S_3\times U_4(3).(2^2)_{122}$     & $7a$ \\
$M_{12}.2\times 2$               & $11a$             & ${}^2F_4(2)'.2$                         & $13a$ \\
                            &                   & $G_2(3).2$                           & $13a$ \\
\underline{$He$}                          &                   &                                   &  \\
$S_4(4).2$                     & $17a$ $17b$       & \underline{$Fi_{24}'$}                             &  \\
$7^{1+2}:(S_3\times 3)$        &  $7d$   $7e$      & $Fi_{23}$                              &  $17a$   $23a$   $23b$  \\
                            &                   & $2^{11}.M_{24}$                             &  $23a$   $23b$  \\
\underline{$He.2$}                        &                   & $2^2.U_6(2).3.2$                     & $11a$ \\
$S_4(4).4$                     &  $17a$            & $He.2$                              &  $17a$  \\
$7^{1+2}:(S_3\times 6)$        & $7c$              & $He.2$                            &  $17a$  \\
                            &                   & $(3^2:2\times G_2(3)).2$             & $13a$ \\
\underline{$HN$}                          &                   & $29:14$                          &  $29a$   $29b$  \\
$A_{12}$                         & $7a$              &                                   &  \\
$2.HS.2 $                     & $11a$             & \underline{$Fi_{24}$}                              &  \\
$U_3(8).3_1$                   &  $19a$   $19b$    & $2\times Fi_{23}$                      &  $17a$   $23a$   $23b$  \\
                            &                   & $2^{12}.M_{24}$                          &  $23a$   $23b$  \\
\underline{$HN.2$}                       &                   & $2^2.U_6(2):S_3\times 2$              & $11a$ \\
$A_{12}.2$                       & $7a$              & $(S_3\times S_3\times G_2(3)):2$       & $13a$ \\
$4.HS.2$                      & $11a$             & $29:22$                              &  $29a$  \\
$U_3(8).6$                     &  $19a$            &                                   &  \\
                            &                   & \underline{$B$}                                 &  \\
\underline{$Th$}                          &                   & $2.{}^2E_6(2).2$                        &  $19a$  \\
$U_3(8).6$                     &  $19a$            & $2^{1+22}.Co_2$                      &  $23a$   $23b$  \\
$(3\times G_2(3)):2$           & $13a$             & $Th$                                &  $31a$   $31b$  \\
$L_2(19).2$                    &  $19a$            & $(2^2 \times  F_4(2)).2$             &  $17a$  \\
$31:15$                   &  $31a$   $31b$    & $S_4\times {}^2F_4(2)$                & $13a$ \\
                            &                   & $A_5.2\times M_{22}.2$                  & $11a$ \\
\underline{$Fi_{22}$}                        &                   & $L_2(31)$                            &  $31a$   $31b$  \\
$2.U_6(2)$                     & $11a$ $11b$       & $47:23$                               & $47a$ $47b$ \\
$O_7(3)$                       & $13a$ $13b$       &                                   &  \\
$O_7(3)$                      & $13a$ $13b$       & \underline{$Ly$}                                &  \\
$O_8^+(2).3.2$                  & $5a$ $7a$         & $G_2(5)$                             &  $31a$--$e$     \\
$S_3\times U_4(3).2_2$          & $7a$              & $3.McL.2$                           & $11a$ $11b$ \\
${}^2F_4(2)'$                     & $13a$ $13b$       & $2.A_{11}$                             & $7a$ \\
$A_{10}.2$                       & $7a$              & $67:22$                         & $67a$--$c$  \\
$A_{10}.2$                       & $7a$              & $37:18$                          & $37a$  $37b$  \\
                            &                   &                                   &  \\
\underline{$Ru$}                          &                   &                                   &  \\
$(2^2\times Sz(8)):3$         & $7a$ $13a$        &                                   &  \\
$L_2(29)$                      &  $29a$   $29b$    &                                   &  \\
\hline
\end{tabular}
    \caption{Almost simple groups $\underline{G}$  with socle a sporadic simple group with maximal subgroups $H$ such that $H$ contains a {\qsre} belonging to class indicated in the class column.}
    \label{tab:spor3}
\end{table}

\end{document}